\documentclass[11pt,leqno,oneside]{amsart}
\usepackage[utf8]{inputenc} %
\usepackage[T1]{fontenc}    %
\usepackage{microtype} 
\usepackage[english]{babel} %
\usepackage{stmaryrd}

\usepackage[a4paper,margin=.9in,marginpar=.9in]{geometry} %
\usepackage{csquotes}       %

\usepackage[shortlabels,inline]{enumitem}

\usepackage{amsmath,amssymb}   %
\usepackage{amsthm}    %
\usepackage{mathtools} %
\usepackage[lcgreekalpha]{stix2}          %

\usepackage[
backend=biber,        %
style=ext-alphabetic,     %
giveninits=true,      %
maxalphanames=5,      %
minalphanames=5,      %
maxbibnames=99,       %
mincitenames=5,
maxcitenames=5,
articlein=false,
useprefix=true,
]{biblatex}

\AtEveryBibitem{%
  \clearfield{note}%
  \clearlist{location}%
  \clearlist{language}%
}

\DeclareSourcemap{
  \maps[datatype=bibtex]{
    \map{
      \step[fieldsource=doi,final]
      \step[fieldset=url,null]
    }  
  }
}

\usepackage{xcolor}                   %
\definecolor{green}{HTML}{2ECC71}
\definecolor{blue}{HTML}{3498DB}
\definecolor{red}{HTML}{E74C3C}
\definecolor{orange}{HTML}{FD6A02}
\usepackage[pdfpagelabels]{hyperref}  %
\hypersetup{%
  colorlinks,
  linktocpage,
  urlcolor  = blue,
  citecolor = blue,
linkcolor = blue}

\usepackage[capitalise,sort]{cleveref} %

\AfterEndEnvironment{proof}{\noindent\ignorespaces} %
\makeatletter
\def\@endtheorem{\endtrivlist}
\makeatother

\Crefname{paragraph}{\S}{\SS}
\Crefname{equation}{}{}
\Crefname{enumi}{}{}
\Crefname{conditioni}{Condition}{Conditions}
\Crefname{conditionalti}{Condition}{Conditions}
\newtheorem{theorem}{Theorem}[section]
\newtheorem*{theorem*}{Theorem}
\Crefname{theorem}{Theorem}{Theorems}

\Crefname{theoremintro}{Theorem}{Theorems}

\newtheorem{lemma}[theorem]{Lemma}
\Crefname{lemma}{Lemma}{Lemmas}
\newtheorem{proposition}[theorem]{Proposition}
\Crefname{proposition}{Proposition}{Propositions}
\newtheorem{corollary}[theorem]{Corollary}
\Crefname{corollary}{Corollary}{Corollaries}

\Crefname{conjecture}{Conjecture}{Conjectures}

\newtheorem{principle}{Principle}

\theoremstyle{definition}

\Crefname{example}{Example}{Examples}
\newtheorem*{example*}{Example}
\Crefname{assumption}{Assumption}{Assumptions}
\newtheorem{definition}[theorem]{Definition}
\Crefname{definition}{Definition}{Definitions}

\Crefname{question}{Question}{Questions}
\theoremstyle{remark}
\newtheorem{remark}[theorem]{Remark}
\Crefname{remark}{Remark}{Remarks}

\numberwithin{equation}{section} %

\usepackage{booktabs} %

\DeclarePairedDelimiter{\paren}{\lparen}{\rparen}
\DeclarePairedDelimiter{\bracket}{\lbrack}{\rbrack}
\DeclarePairedDelimiter{\set}{\lbrace}{\rbrace}
\DeclarePairedDelimiter{\abs}{\lvert}{\rvert}

\DeclarePairedDelimiter{\norm}{\lVert}{\rVert}
\DeclarePairedDelimiterX{\psh}[2]{\langle}{\rangle}{#1, #2}
\DeclarePairedDelimiterX{\pairing}[2]{\langle}{\rangle}{#1 \vert #2}
\DeclarePairedDelimiter{\floor}{\lfloor}{\rfloor}

\DeclarePairedDelimiterXPP{\Exp}[1]{\exp}{\lparen}{\rparen}{}{#1}
\DeclarePairedDelimiterXPP{\Log}[1]{\log}{\lparen}{\rparen}{}{#1}
\DeclarePairedDelimiterXPP{\Inf}[1]{\inf}{\lbrace}{\rbrace}{}{#1}
\DeclarePairedDelimiterXPP{\Sup}[1]{\sup}{\lbrace}{\rbrace}{}{#1}
\DeclarePairedDelimiterXPP{\Max}[1]{\max}{\lbrace}{\rbrace}{}{#1}
\DeclarePairedDelimiterXPP{\Min}[1]{\min}{\lbrace}{\rbrace}{}{#1}

\renewcommand{\L}{\mathsf{L}}
\newcommand{\I}{\mathsf{I}}
\newcommand{\E}{\esp}

\DeclareMathOperator{\esp}{\mathbf{E}}
\DeclareMathOperator{\prob}{\mathbf{P}}
\DeclareMathOperator{\var}{\mathbf{Var}}
\DeclareMathOperator{\law}{\mathbf{law}}
\DeclareMathOperator{\cov}{\mathbf{cov}}

\DeclarePairedDelimiterXPP{\Prob}[1]{\prob}[]{}{#1}
\DeclarePairedDelimiterXPP{\Esp}[1]{\esp}[]{}{#1}
\DeclarePairedDelimiterXPP{\Var}[1]{\var}[]{}{#1}
\DeclarePairedDelimiterXPP{\Law}[1]{\law}[]{}{#1}
\DeclarePairedDelimiterXPP{\Cov}[2]{\cov}[]{}{#1, #2}

\author[R.\ Herry]{Ronan HERRY}
\address{IRMAR, Université de Rennes 1}
\email{ronan.herry@univ-rennes.fr}
\urladdr{https://orcid.org/0000-0001-6313-1372}
\author[D.\ Malicet]{Dominique MALICET}
\address{LAMA, Université Gustave Eiffel}
\email{dominique.malicet@univ-eiffel.fr}
\urladdr{https://orcid.org/0000-0003-2768-0125}
\author[G.\ Poly]{Guillaume POLY}
\address{IRMAR, Université de Rennes 1}
\email{guillaume.poly@univ-rennes.fr}

\title[Limit distributions for polynomials]{Limit distributions for polynomials with independent and identically distributed entries}

\begin{document}

\begin{abstract}
  We characterize the limiting distributions of random variables of the form \( P_n\left( (X_i)_{i \ge 1} \right) \), where: \begin{enumerate*}[(i)]
    \item \( (P_n)_{n \ge 1} \) is a sequence of multivariate polynomials, each potentially involving countably many variables;
    \item there exists a constant \( D \ge 1 \) such that for all \( n \ge 1 \), the degree of \( P_n \) is bounded above by \( D \); 
    \item \( (X_i)_{i \ge 1} \) is a sequence of independent and identically distributed random variables, each with zero mean, unit variance, and finite moments of all orders.
  \end{enumerate*}
  More specifically, we prove that the limiting distributions of these random variables can always be represented as the law of \( P_\infty\left( (X_i, G_i)_{i \ge 1} \right) \), where \( P_\infty \) is a polynomial of degree at most \( D \) (potentially involving countably many variables), and \( (G_i)_{i \ge 1} \) is a sequence of independent standard Gaussian random variables, which is independent of \( (X_i)_{i \ge 1} \).

  The characterization of all possible limiting laws of polynomials in independent and identically distributed variables is a long-standing problem, that we trace back at least to Kolmogorov's influential school in Probability in the 1960s.
  The seminal work \cite{Sevastyanov} is the first to solve this problem for $D=2$ and when the $(X_{i})$ are Gaussian.
  There, a diagonalization argument serves as the main analytical tool.
 In contrast, the case of non-Gaussian inputs has been solved only recently:
 \cite{BDM,BMMBernoulli} propose a solution for non-Gaussian quadratic polynomials, \(\sum_{i,j \leq N_n} \alpha_n(i,j) X_i X_j\), where the common law of the $(X_{i})$ is generic but the coefficients \((\alpha_n(i,j))_{i,j \leq N_n}\) form an adjacency matrix.
 This extra assumption enables combinatorial arguments grounded in graph-theoretic techniques. 

We solve this problem in full generality, addressing both Gaussian and non-Gaussian inputs, and with no extra assumption on the coefficients of the polynomials.
In the Gaussian case, our proof builds upon several original tools of independent interest, including a new criterion for central convergence based on the concept of \emph{maximal directional influence}.
Beyond asymptotic normality, this novel notion also enables us to derive quantitative bounds on the degree of the polynomial representing the limiting law.
We further develop techniques regarding asymptotic independence and dimensional reduction.
To conclude for polynomials with non-Gaussian inputs, we combine our findings in the Gaussian case with invariance principles from \cite{MOO}.
\end{abstract}

\maketitle%
\tableofcontents%

\section{Introduction}

\subsection{Main results}  
We establish a complete characterization of the closure, for the topology of convergence in law, of polynomials in independent random variables that are centered, have unit variance, and possess finite moments of all orders.
\subsubsection{Stability in law of polynomials chaoses}
We present below our main result, restricted to the univariate setting.

\begin{theorem}[Stability of polynomial chaoses]\label{th:stabilite-chaos}  
  Let \( d \geq 1 \), and let \((P_n)\) be a sequence of multivariate real polynomials of degree at most \( d \in \mathbb{N} \). Consider \( \vec{X} = (X_i)_{i \geq 1} \), a sequence of independent and identically distributed random variables satisfying \( \Esp{X_1} = 0 \), \( \Esp{X_1^2} = 1 \), and \( \Esp{\abs{X_1}^p} < \infty \) for all \( p \in \mathbb{N} \).
Assume that \( P_n(\vec{X}) \) converges in law to a limit \( \mu \). Then, there exist:  
\begin{itemize}[wide]
    \item \( \vec{G} = (G_i)_{i \geq 1} \), a sequence of independent standard Gaussian random variables, also independent of \( \vec{X} \);
    \item \( P_\infty \), a multivariate polynomial of degree at most \( d \), possibly involving countably many variables,  
\end{itemize}  
such that:  
\begin{equation*}
  \Law*{P_\infty(\vec{X}, \vec{G})} = \mu.
\end{equation*}
\end{theorem}

As a special case, we investigate polynomials in Gaussian variables, that is \( \Law{X_1} = \gamma \).
In this setting, we derive the following result, which extends a prior result known only for \( d \leq 2 \) (see \cite{Sevastyanov}).

\begin{theorem}[Stability of Wiener chaoses]\label{th:stabilite-wiener}  
Let \( \vec{G} = (G_i) \) be a sequence of independent standard Gaussian random variables. The set  
\begin{equation*}
    \set*{ P(\vec{G}) : P \text{ is a polynomial of degree at most \( d \)} }
\end{equation*}
is closed for the topology of convergence in distribution.
\end{theorem}

In \cref{th:stabilite-general-multidim}, we actually derive a stronger result than \cref{th:stabilite-chaos} in two ways:
\begin{enumerate}[(i)]
  \item we obtain a multivariate version of the theorem, that is we can consider polynomial random vectors $(P_{j}(\vec{X}))_{j \in \mathbb{N}}$ where each $P_{j}$ is a polynomial of degree at most $d$;
  \item the $X_{i}$'s are not necessarily identically distributed, we only need to assume that they have finite moments and that $\set*{ \Law{X_{i}} : i \in \mathbb{N} }$ is finite.
\end{enumerate}
Similarly, in \cref{th:stabilite-wiener-multidim}, we state the multivariate counterpart of \cref{th:stabilite-wiener}.

\subsubsection{Directional influences}
To prove these two results, we introduce a novel quantity, the \emph{directional influence of order \( k \)}.
For simplicity, in this introduction, we restrict the presentation of this object to homogeneous polynomials without \textit{diagonal terms} with Gaussian entries, and involving finitely many variables, which suffices to illustrate the essence of our results.
Namely, for \( F \), a homogeneous polynomial of degree \( d \) of the form  
\begin{equation*}
F = \sum_{i_1 < \cdots < i_d \leq K} a_{i_1, \ldots, i_d} G_{i_1} \cdots G_{i_d},
\end{equation*}
writing $\partial_{G_{i}}$ for the partial derivate with respect to $G_{i}$, we define the semi-norms:  
\begin{equation*}
\tilde{\rho}_k(F) \coloneq  \Sup*{ \norm*{ \sum_{i \geq 1} \partial_{G_i} F \, \partial_{G_i} X }_{L^{2}} : X = \sum_{i_1 < \cdots < i_d \leq K} b_{i_1, \ldots, i_d} G_{i_1} \cdots G_{i_d},  \norm{X}_{L^{2}} \leq 1 }, \qquad k \in \mathbb{N}.
\end{equation*}
Using this new framework, we establish a necessary and sufficient criterion for central convergence, which plays a pivotal role in our analysis.

\begin{theorem}[Asymptotic normality from directional influences]\label{th:normal-convergence}  
Let \( d \geq 1 \) and consider a sequence
\begin{equation*}
  F_n \coloneq \sum_{i_1 < \cdots < i_d} a_n(i_1, \ldots, i_d) G_{i_1} \cdots G_{i_d}, \qquad n \in \mathbb{N}.
\end{equation*}
Then, the following are equivalent:
\begin{enumerate}[(i)]
  \item $(F_{n})$ is asymptotically normal.
  \item For some $k \in \set*{ \floor*{\frac{d}{2}}, \dots, d-1}$, we have
\begin{equation*}
 \tilde{\rho}_k(F_n) \to 0.
\end{equation*}
\end{enumerate}
\end{theorem}

See \cref{th:wiener:normal-convergence} for a precise statement.
With respect to the problem of characterizing all the distributional limits, this criterion enables, starting from a non asymptotically Gaussian polynomials, to factor out polynomials of smaller degrees, allowing for an inductive argument.
Numerous alternative criteria for central convergence exist, most notably:
\begin{enumerate}[(i)]
  \item the celebrated Fourth-Moment Theorem of \citeauthor{NualartPecattiCLT}~\cite{NualartPecattiCLT};
  \item a criterion based Malliavin calculus by \Citeauthor{NualartOrtizLatorre}~\cite{NualartOrtizLatorre};
\item the Second Order Poincaré inequalities developed by \Citeauthor{ChatterjeeSecondOrder}~\cite{ChatterjeeSecondOrder}, and \Citeauthor{NPRPoincare}~\cite{NPRPoincare}.
\end{enumerate}
All the results above give necessary and sufficient conditions for asymptotic normality of homogeneous polynomials without diagonal terms.
Arguably, our criterion is a new contribution in the vast literature of central limit criteria for Gaussian polynomials.
However, directional influences offer new insights in two directions.
\begin{enumerate}[(i)]
  \item For Gaussian limits, we obtain the criterion from \cref{th:normal-convergence}. 
    This plays a crucial role in our proof, we do not know how to derive from the classical criteria.
  \item Beyond Gaussian limits, the semi-norms \( \tilde{\rho}_k \) encode precise information about the distributional convergence of a sequence of polynomials \( (F_n)_{n \geq 1} \), as demonstrated by the following result.
\end{enumerate}

\begin{theorem}\label{th:super-lemme}
  With $(F_{n})$ as above, let \( s \in \set*{ 1,\dots, d-1 } \) and suppose \( \tilde{\rho}_s(F_n) \to 0 \), where
  \begin{equation*}
    F_n \coloneq \sum_{i_1 < \cdots < i_d} a_n(i_1, \ldots, i_d) G_{i_1} \cdots G_{i_d}, \qquad n \in \mathbb{N}.
  \end{equation*}
  Further, assume \( (F_n)_{n \geq 1} \) converges in distribution to \( \mu \). Then there exists a multivariate polynomial \( Q \) such that:  
\begin{equation*}
\Law*{Q(\vec{G})} = \mu, \quad \text{and} \quad \deg(Q) \leq \left\lfloor \frac{d}{s+1} \right\rfloor.
\end{equation*}
\end{theorem}

\begin{remark}
  For $s$ as in \cref{th:normal-convergence}, we have $\floor*{\frac{d}{s+1}} \leq 1$, and thus $Q(\vec{G})$ is Gaussian.
  Thus, \cref{th:super-lemme} contains \cref{th:normal-convergence}.
\end{remark}

\subsection{Motivation}

Early fundamental results by the founding figures of the theory of probability, such as \Citeauthor{Chebyshev}~\cite{Chebyshev}, \Citeauthor{Lindeberg}~\cite{Lindeberg}, \Citeauthor{Kolmogorov}~\cite{Kolmogorov}, \Citeauthor{Levy}~\cite{Levy}, and \Citeauthor{Feller}~\cite{Feller} give a complete characterization of the limits in law of the sum of $n$ independent random variables, as $n \to \infty$.
We refer to the monographs \cite{LevyBook,GnedenkoKolmogorov} for thorough introductions to the subject and more references.

Polynomials evaluated in random variables are central in several branches of probability theory: they allow to model complex stochastic behaviour, and decomposing a random field on a polynomial basis is a common strategy to study a probabilistic model.
Such decomposition is known under various names: Walsh decomposition for boolean functions \cite{Walsh}, Wiener's \emph{polynomial chaos} \cite{WienerChaos}, and its discrete counterpart \cite{WienerDiscreteChaos}, Itô's \emph{multiple integrals decomposition} for Gaussian fields \cite{ItoWienerIntegrals} and for Poisson fields \cite{ItoPoissonIntegrals}, or yet Hoeffding's \emph{ANOVA decomposition} \cite{Hoeffding}.
To illustrate the importance of this type of decomposition, here is a short non-exhaustive compilation of works, pertaining to different areas of mathematics, fruitfully exploiting polynomials with random entries.
For conciseness, we restrict to papers from less than ten years ago (that is after 2015): \cite{MarinucciPeccatiRossiWigman,BhattacharyaDiaconisMukerjee,DeServedio,KindlerKirshnerODonnell,CaravennaSunZygourasKPZ,EskenazisIvanisvili,AdamczakLatalaMeller,ArmentanoAzaisDalmaoLeon,JaramilloNourdinPeccati,CaravennaSunZygourasHeatFlow,AdamczakPivovarovSimanjuntak,HairerRenormalisation,GloriaQi}.
We also mention that polynomial decomposition plays an important roles in computational mathematics, see for instance the influential papers \cite{XiuKarniadakis,BlatmanSudret}.

In view of the importance of polynomials with random inputs, following the resolution of the linear case, subsequent works have intensively studied probabilistic properties of such polynomials.
We refer to the two excellent surveys by \citeauthor{BogachevSurvey}, and the references therein, for more details on polynomials evaluated in generic random variables \cite{BogachevPolynomials}, and in independent Gaussian variables \cite{BogachevSurvey}.

The primary motivation of this work is to address the long-standing problem of characterizing the limiting distributions of polynomials, at least when the inputs admit finite moments.
As with any characterization result, this advances the field not only by providing precise insights on the behavior of random polynomials, but also by developing original tools to tame their complexity.
In particular, our novel analysis based on the directional influences allows us to break down a polynomials in pieces of lower degree.
Each of these smaller pieces can be further dissected until we obtain a decomposition for which we have an explicit control on each of the terms.
We believe that such a procedure could turn useful in other problems related to polynomials in random variables, as it enables an efficient dimension reduction.

\subsection{Related works}

\subsubsection{Characterization of the limits}

The problem we solve in this paper is explicitly formulated by \Citeauthor{KolmogorovQuestion}~\cite{KolmogorovQuestion} during a seminar, at least in the case of Gaussian inputs and degree $2$, but could have appeared earlier.
Quickly thereafter \cite{Sevastyanov} presents a solution for degree $2$ Gaussian polynomials, that is the case $d=2$ of our \cref{th:stabilite-wiener}.
This solution later reappears in \cite{Arcones}.
Despite being mentioned in different sources \cite[p.~85]{Janson}, \cite[Problem~4, p.~740]{BogachevPolynomials}, \cite[Question~3, p.~562]{BogachevSurvey}, \cite[Conclusion]{BDM}, before the present work, this question has only received very partial answers.
\begin{itemize}[wide]
  \item \cite[Cor.~2]{AgafontsevBogachev} provides a form of almost sure stability for Gaussian polynomials of a certain type.
  \item Relying on a diagonalisation argument, \cite[Thm.~1.2]{BogachevKosovNourdinPoly} proves a version of \cref{th:stabilite-wiener} for $d =2$ and for multivariate vectors.
  \item \cite{BMMBernoulli,BDM} gives a characterisation of the limits of generic degree $2$ polynomials under the additional constraints that the coefficients of the polynomials are in $\{0,1\}$, allowing for graph-theoretic argument.
\end{itemize}
Apart from these works, we are not aware of any prior conclusive results addressing the stability of the laws of polynomials with random inputs.

Aside from the aforementioned characterisation of the limits, polynomials in random variables have been intensively studied in many directions \cite{BogachevSurvey,BogachevPolynomials}.
In this section, we highlight two particular directions of research that are indirectly connected to the question of the stability of distributions of polynomials.

\subsubsection{Central limit theorems}
An important line of research focuses on generalizing the central limit theorem to non-linear functionals, particularly polynomial ones.
In this context, the generalization typically involves identifying sufficient — and sometimes necessary — conditions to guarantee asymptotic normality, along with quantifying the rate of convergence in appropriate probabilistic distances.
Criteria ensuring central convergence serve as the backbone of asymptotic results for polynomials in independent and identically distributed random variables and play a pivotal role in our approach.
A crucial ingredient of our proof is the recursive decomposition of a polynomial into lower-degree ones that exhibit asymptotic Gaussian behavior.

Let us quote, in a very non-exhaustive way, some seminal contributions regarding non-linear central limit theorems.
The interested reader can also consult the references therein.
\begin{itemize}[wide]
  \item In his seminal works, \Citeauthor{deJongQuadratic} gives sufficient conditions for central limit theorems for quadratic forms \cite{deJongQuadratic} and for multilinear polynomials \cite{deJongPolylinear}.
  \item \Citeauthor{NualartPecattiCLT}~\cite{NualartPecattiCLT}, and \Citeauthor{PeccatiTudor}~\cite{PeccatiTudor}, characterize central limit theorems for homogeneous Gaussian polynomials, also known as \emph{Wiener chaoses}, in term of the convergence of the fourth moment.
  Still, in the Gaussian setting, \Citeauthor{NualartOrtizLatorre}~\cite{NualartOrtizLatorre} express condition for asymptotic normality in terms of operators from Malliavin calculus.
 Later on, \Citeauthor{NourdinPeccatiStein}~\cite{NourdinPeccatiStein} use similar ideas from Malliavin calculus, and combine them with Stein's method, to quantity normal convergence in Kolmogorov distance and total variation distance.
 See \cite{NourdinPeccatiBlueBook,AzmoodehPeccatiYang} for further developments in this line of research.
  \item \Citeauthor{ChatterjeeNewMethod}~\cite{ChatterjeeNewMethod} builds upon the celebrated Stein's method to obtain a new criterion for asymptotic normality in terms of a variance bound.
  In the same spirit, \Citeauthor{ChatterjeeSecondOrder}~\cite{ChatterjeeSecondOrder}, and \Citeauthor{NPRPoincare}~\cite{NPRPoincare} provide criteria for asymptotic convergence of non-linear functionals of a Gaussian fields, that can actually be used beyond the polynomial setting.
\item \Citeauthor{DoeblerPeccatiPoisson}~\cite{DoeblerPeccatiPoisson} extends the fourth-moment theorem to case of polynomials with Poisson entries, also known as the Poisson--Wiener chaoses.
  We stress that our \cref{th:stabilite-chaos} does not apply in the full generality of Poisson chaoses.
  Contrarily to the Gaussian case, a Poisson distribution of mean $\lambda$ cannot be expressed as a polynomial transform of Poisson distribution with mean $1$.
  Thus, even the Poisson chaos of degree one, cannot realised as linear form in a independent and identically distributed sequence.
  It is a question of high interest to determine the closure in law of Poisson chaoses.
\end{itemize}

\subsubsection{Invariance principles}
A second popular line of research concerns the so called \emph{invariance principles}.
Informally speaking, invariance principles compare the law of a polynomial $P(\vec{X})$, where $\vec{X} = (X_{i})$ is a sequence of independent random variables with a common law that is rather generic, to that of $P(\vec{G})$ where $\vec{G}$ is a sequence of independent standard Gaussian.
Such invariance principles allow to generalize results and tools from the Gaussian framework to a more singular one.
In particular, they play a prominent role in our proof of the stability theorem for polynomials with generic inputs.

As above, we present a short and non-exhaustive selection of emblematic results in the field.
\begin{itemize}[wide]
  \item 
\Citeauthor{RotarQuadratic} gives sufficient conditions for invariance principles for quadratic forms in \cite{RotarQuadratic}, and for multilinear polynomials in \cite{RotarPolylinear}.
  Later, \Citeauthor{GotzeTikhomirov}~\cite{GotzeTikhomirov} quantify Rotar's invariance principle.%
\item \Citeauthor{ChatterjeeLindeberg}~\cite{ChatterjeeSimpleInvariance,ChatterjeeLindeberg} establishes a new invariance principle beyond the polynomial setting, encompassing smooth functionals.
\item \Citeauthor{MOO}~\cite{MOO} proves an invariance principle for polynomials, with an explicit quantification in the degree and the so-called \emph{influence}.
  This principle allows for the resolution of two important conjectures related to boolean functions theory.
\Citeauthor{NPR}~\cite{NPR} builds upon the aforementioned invariance principle, and derive a central limit theorems for a polynomials with general random independent entries from the particular case of Gaussian entries.
\end{itemize}

\subsection{Outline of the proof in the Gaussian case}

\subsubsection{A new criterion for asymptotic normality}
As previously mentioned, our proof of \cref{th:stabilite-wiener} relies on the following principle, which characterizes the asymptotic normality of Gaussian polynomials.
See \cref{th:normal-convergence} and \cref{th:wiener:normal-convergence} for more precise statements.

\begin{principle}\label{principle}
  Let $\vec{G}$ be an infinite vector of independent standard Gaussian variables, and let $(P_n)$ be a sequence of multilinear polynomials of degree $d \geq 2$ without diagonal terms (that is, with no repeated indices).
  Define \( F_n \coloneq P_n(\vec{G}) \), which is assumed to be centered with unit variance.
  Then, the following statements are \emph{heuristically} equivalent:
  \begin{enumerate}[(i)]
    \item $(F_n)_{n \geq 1}$ is asymptotically Gaussian.
    \item $(F_n)_{n \geq 1}$ is asymptotically independent of all polynomials in $\vec{G}$ of degree at most $d-1$.
    \item $(F_n)_{n \geq 1}$ is asymptotically independent of all polynomials in $\vec{G}$ of degree at most $\floor*{\frac{d}{2}}$.
  \end{enumerate}
\end{principle}

We quantify the asymptotic independence of $(F_{n})$ with a sequence of polynomials $(X_{n})$ in terms of
\begin{equation*}
  \Gamma(F_{n}, X_{n}) \coloneq \sum_{i \in \mathbb{N}} \partial_{G_{i}} F_{n} \partial_{G_{i}} X_{n}.
\end{equation*}
The seminorm $\tilde{\rho}_{k}$ rewrites in terms of $\Gamma$, and in \cref{th:wiener:asymptotic-independence}, we explicitly connect the vanishing of the seminorm with some asymptotic independence property.

For convenience, in the sequel, we present a slightly modified version of this semi-norm, \( \rho_k(F) \), where the supremum is taken over the unit ball of \( \mathcal{W}_d \), the Wiener chaos of order \( d \), that is a distinguished subset of degree $d$ polynomials.
Since \( \tilde{\rho}_k \) and \( \rho_k \) are bi-comparable, we do not delve into these refinements, which require defining Wiener chaoses.

For \( d \in \{2,3\} \), \( \lfloor d/2 \rfloor = 1 \), and the criterion for asymptotic normality takes a simpler form:
\begin{equation*}
F_n \to \mathcal{N}(0,1) \,\Leftrightarrow\, \tilde{\rho}_1(F_n) \to 0 \,\Leftrightarrow\, \sup_{\vec{a} \in l^2(\mathbb{N}), \|\vec{a}\|_2 = 1} \left\| \sum_i a_i \partial_{G_i} F_n \right\|_2 \to 0.
\end{equation*}

To illustrate our proof, we restrict ourselves to the cases \( d \in \{1,2,3\} \), which are sufficiently rich to highlight the mechanisms of the proof and are of independent interest.

\subsubsection{Schematic proof: factoring out influential directions}
Our approach proceeds by induction and relies on \cref{principle} in its contrapositive form. Assume that \( d \in \{1,2,3\} \) and that \( F_n \) is not asymptotically Gaussian. Then, there exists \( \vec{a}_n \in l^2(\mathbb{N}) \) and \( \delta > 0 \) such that 
\[
  \norm*{\sum_{i} a_{n,i}\partial_{G_{i}} F_{n}}_{L^{2}(\prob)}  \geq \delta,
\]
for sufficiently large \( n \). Setting \( H_n \coloneq \sum_{i \geq 1} a_n(i) G_i \sim \mathcal{N}(0,1) \), this implies that \( F_n \) depends macroscopically on \( H_n \). 

By performing an orthonormal change of basis such that \( G_1 \to H_n \), which preserves the distribution of $\vec{G}$, one can write
\[
F_n = \sum_{k=1}^d A_{k,n}(G_2, \ldots) H_k(G_1) + A_{0,n},
\]
where the coefficients \( (A_{k,n})_{0 \leq k \leq d} \) are polynomials in \( (G_i)_{i \geq 2} \), hence independent of \( G_1 \). The coefficients \( (A_{k,n})_{n \geq 1} \) have degree strictly less than \( d \) and can be handled inductively, while \( A_{0,n} \) remains of degree \( d \).

Repeating this process for \( A_{0,n} \), and after a sufficient number of iterations, the remainders become asymptotically Gaussian, as per the aforementioned principle. Specifically, for some appropriately chosen \( r_n \to \infty \), we write:
\[
F_n = \underbrace{\left(F_{n,1} + \cdots + F_{n,r_n}\right)}_{\text{Induction Step}} + \underbrace{R_{n,r_n},}_{\text{Gaussian Remainder}}
\]
where the following points are noteworthy:
\begin{itemize}
    \item[(i)] The term \( F_{n,1} + \cdots + F_{n,r_n} \) involves countably many polynomials of degree strictly less than \( d \), requiring the induction to be carried out in an infinite-dimensional setting.
    \item[(ii)] It is necessary to exchange the limits in \( r \) and \( n \) for the series \( F_{n,1} + \cdots + F_{n,r} \). By construction, \( \mathbb{E}(F_{n,i} F_{n,j}) = 0 \) for \( i \neq j \), allowing for \( L^2 \)-control of the series.
    \item[(iii)] For \( d \geq 4 \), the criterion involving \( \tilde{\rho}_1 \) is insufficient. However, the same strategy applies, with decompositions of \( (F_n)_{n \geq 1} \) taking the form 
    \[
    F_n = \sum_{k=0}^{\lfloor d/q \rfloor} A_{k,n} H_k(X_n),
    \]
    where \( (X_n) \) is a suitable polynomial of degree \( q < d \) that is asymptotically normal, and \( (A_{k,n})_{k \geq 0} \) are polynomials asymptotically independent of \( (X_n)_{n \geq 1} \). Justifying such decompositions and asymptotic independence is non-trivial; see \cref{s:proof-wiener} for details.
\end{itemize}
\subsubsection{The statement for low degree polynomials}
We give a \emph{ad hoc} definition of Wiener chaoses.
We give more remainders on Wiener chaoses in \cref{s:wiener:reminder}.
We fix $\vec{G} = (G_{i} : i \in \mathbb{N})$ a sequence of independent standard Gaussian variables, and we write $H_{k}$ for the Hermite polynomial of degree $k$.
Here by $\operatorname{vect} A$, we mean the closure in $L^{2}(\prob)$ of the linear space generated by $A$ and we set
\begin{equation}\label{eq:def-wiener-chaos}
  \mathcal{W}_{p} \coloneq \operatorname{vect} \set*{ \prod_{i \in \mathbb{N}} H_{k_{i}}(G_{i}) : \sum_{i} k_{i} = p }.
\end{equation}
In particular, $\mathcal{W}_{0} = \mathbb{R}$, and $\mathcal{W}_{1}$ contains only Gaussian variables, and each $\mathcal{W}_{p}$ in closed in $L^{2}(\prob)$.
We also recall that
\begin{equation}\label{eq:wiener-decomposition}
  \mathcal{W}_{\leq p} \coloneq \bigoplus_{p=0}^{d} \mathcal{W}_{p} = \set*{ P(\vec{G}) : \operatorname{deg} P \leq d }.
\end{equation}
Our induction is formulated as follows.
\begin{definition}
  For $d \in \mathbb{N}^{*}$, we say that a sequence $(\vec{F}_{n})$ of random infinite vectors is $d$-admissible provided that:
  \begin{align}
    & \exists K \in \mathbb{N}^{*},\, \forall i \in \mathbb{N},\, \exists p_{i} \in \{0,\dots, K\},\, \forall n \in \mathbb{N},\, \bracket*{ F_{n,i} \in \mathcal{W}_{p_{i}},\ \text{and}\ \Esp*{F_{n,i}^{2}} = 1 };
  \\& \forall i \in \mathbb{N},\, \bracket*{ p_{i} \leq d,\ \text{or}\ F_{n,i} \xrightarrow[n \to \infty]{\law} \mathcal{N}(0,1) }.
  \end{align}
\end{definition}
Let us establish the following result, which in view of \cref{eq:wiener-decomposition} implies \cref{th:stabilite-wiener} for $d \in \{1,2,3\}$.
\begin{theorem}\label{th:wiener:stabilite:outline}
  Let $(\vec{F}_{n})$ be a $d$-admissible vector with $d \in \{1,2,3\}$, that converges in law to some $\vec{F}_{\infty}$.
  Then, there exists a sequence $(Y_{i}) \in \mathcal{W}_{\leq d}^{\mathbb{N}}$ such that
  \begin{equation*}
    \law{\vec{F}_{\infty}} = \law{\vec{Y}}.
  \end{equation*}
\end{theorem}
\subsubsection{Initialisation: proof of the case \texorpdfstring{$d=1$}{d=1}}
Let us consider a $1$-admissible sequence $(\vec{F}_{n})$ converging in law to some $\vec{F}_{\infty}$.
We can assume that $\vec{F}_{n}$ has no deterministic component, that is $p_{i} > 0$ for all $i \in \mathbb{N}$.
Since we work with Gaussian polynomials, convergence in law implies convergence of moments, see \cref{th:polynomials:convergence-moment}, we have that
\begin{equation*}
  \Esp*{ F_{n,i} F_{n,j}} \xrightarrow[n \to \infty]{} c_{ij}.
\end{equation*}
The first step consists of exhibiting $\vec{N}=(N_1,N_2,\cdots)$ a sequence in $\mathcal{W}_1$ such that $\E(N_i N_j)=c_{i,j}$. Such a sequence can be constructed inductively. Assume that we have built $(N_1,\cdots,N_p)$ whose covariance matrix is $C_p\coloneq(c_{i,j})_{1\le i,j \le p}$. We seek for $N_{p+1}=\sum_{i=1}^p \alpha_i N_i +\alpha_{p+1} G$ where $G$ is independent of $(N_1,\cdots,N_p)$, whenever $C_p$ is invertible. If $C_p$ is not invertible, we can build $N_{p+1}$ as a linear combination of $G$ and $(N_{i_1},\cdots,N_{i_q})$ which is of full rank among $(N_1,\cdots,N_p)$. We deal only with case $\det(C_p)\neq 0$ below, for simplicity. The coefficients $(\alpha_1,\cdots,\alpha_{p+1})$ must fulfil
\begin{enumerate}
\item $\E\left(N_{p+1} N_i\right)=c_{i,p+1}, i\in \llbracket 1,p\rrbracket$ which gives 
$
\begin{pmatrix}
c_{1,1} & \cdots & c_{1,p}\\
\vdots & \vdots & \vdots \\
c_{p,1}& \cdots & c_{p,p}\\
\end{pmatrix}
\begin{pmatrix}
\alpha_1\\
\vdots\\
\alpha_p
\end{pmatrix}
=
\begin{pmatrix}
c_{1,p+1}\\
\vdots\\
c_{p,p+1}\\
\end{pmatrix}
.
$
This gives $(\alpha_1,\cdots,\alpha_p)$ by inverting the system. 
\item Besides $\E\left(N_{p+1}^2\right)=c_{p+1,p+1}=\alpha_{p+1}^2+\sum_{1\le i,j\le p} \alpha_i \alpha_j c_{i,j}$ which provides $\alpha_{p+1}$ up to the sign, which does not matter. The resulting sequence $(N_i)_{i\ge 1}$ admits $C=(c_{i,j})_{1\le i,j \le p}$ as covariance matrix.
\end{enumerate}

A seminal result of \citeauthor{PeccatiTudor}~\cite{PeccatiTudor} asserts that convergence in distribution of vectors with Wiener chaotic entries to Gaussian vectors is equivalent to convergence of covariances matrices and component-wise convergence. Hence, we have that
\begin{equation*}
  \vec{F}_{n} \xrightarrow[n \to \infty]{\law} \vec{N},
\end{equation*}
which concludes the proof for $1$-admissible sequences.
\hfill\qedsymbol

\subsubsection{First induction step: from linear to quadratic}
We now show how to deduce the claim for $2$-admissible sequences.
Thus, take $(\vec{F}_{n})$ a $2$-admissible sequence.
As highlighted above, our pivotal idea is that we can factor out Gaussian directions whenever a coordinate fails to be asymptotically Gaussian.
To make this intuition precise, we use the formalism of \emph{Malliavin derivative}, also known as \emph{carré du champ}.
For a precise definition, see \cref{s:wiener:reminder}, here we simply recall that Gaussian polynomials $F = P(\vec{G})$ and $\tilde{F} = \tilde{P}(\vec{G})$, we have
\begin{equation*}
  \Gamma(F, \tilde{F}) \coloneq \sum_{i \in \mathbb{N}} \partial_{G_{i}} F \partial_{G_{i}} \tilde{F}.
\end{equation*}
Let us define \emph{directional influence of degree $1$}
\begin{equation*}
  \rho_{1}(F) \coloneq \Sup*{ \norm{ \Gamma(F, X) }_{L^{2}(\prob)} : X \in \mathcal{W}_{1},\, \norm{X} \leq 1 }.
\end{equation*}
Since every $X \in \mathcal{W}$ is of the form $X = \sum a_{i} G_{i}$ for some $a \in \ell^{2}(\mathbb{N})$, we have the simpler formula
\begin{equation*}
  \Gamma(F, X) = \sum_{i} a_{i} \partial_{i} P(\vec{G}).
\end{equation*}
See \cref{s:influences} for a more detailed discussion on $\rho_{1}$ and related objects.
Write $I$ for the set of indices $i$ such that $F_{n,i}$ is not asymptotically Gaussian.
Then we consider the following iterative construction.
Set $R_{n,i,0} \coloneq F_{n,i}$.
Then assuming that we have constructed $R_{n,i,k}$, define $R_{n,i,k+1}$ in the following way.
Consider $(X_{n,i,k+1})_{n} \in \mathcal{W}_{1}^{\mathbb{N}}$ such that
\begin{equation*}
  \rho_{1}(R_{n,i,k}) \leq \norm{\Gamma(R_{n,i,k}, X_{n,i,k+1})} + \frac{1}{n}.
\end{equation*}
Then, completing $X_{n,i,k+1}$ in an orthonormal Gaussian basis of $\mathcal{W}_1$ and expressing $R_{n,i,k}$ on this basis allows to write:
\begin{equation*}
  R_{n,i,k} = A_{n,i,k+1,2} H_{2}(X_{n,i,k+1})  + A_{n,i,k+1,1} H_{1}(X_{n,i,k+1}) + R_{n,i,k+1},
\end{equation*}
where $A_{n,i,k+1,2} \in \mathcal{W}_{0}$ (it is a constant), $A_{n,i,k+1,1} \in \mathcal{W}_{1}$ and is independent of $X_{n,i,k+1}$, and $R_{n,i,k+1} \in \mathcal{W}_{2}$ is independent of $X_{n,i,k+1}$.
Necessarily, we also find that $X_{n,i,k+1}$ is independent of $X_{n,i,k}$.
Indeed, write
\begin{equation*}
  X_{n,i,k+1} = \sqrt{t} X_{n,i,k} + \sqrt{1-t} N,
\end{equation*}
for some $N \in \mathcal{W}_{1}$ independent of $X_{n,i,k}$, and we have
\begin{equation*}
  \Gamma(X_{n,i,k+1}, R_{n,i,k}) = \sqrt{1-t} \Gamma(N, R_{n,i,k}),
\end{equation*}
which contradicts the almost optimality of $X_{n,i,k+1}$ unless $t=0$.
Thus defining
\begin{equation*}
  F_{n,i,k} \coloneq A_{n,i,k,2} H_{2}(X_{n,i,k}) + A_{n,i,k,1} H_{1}(X_{n,i,k}),
\end{equation*}
we find that
\begin{equation}\label{eq:intro:decomposition-F}
  F_{n,i} = \sum_{k=1}^{l} F_{n,i,k} + R_{n,i,l}.
\end{equation}
Up to extracting a subsequence, we can assume that
\begin{equation*}
  \Esp*{ F_{n,i,k}^{2} } \xrightarrow[n \to \infty]{} v_{i,k},
\end{equation*}
for some non negative numbers $v_{i,k}$'s and any $i,k\ge 1$.
Based on elementary analytical considerations, see \cref{th:analytical-lemma}, one can find a sequence $l_{n} \to \infty$, such that
\begin{equation}\label{eq:intro:convergence-max}
  \sum_{i \leq l_{n}} \sum_{k \leq l_{n}} \abs*{ \Esp*{ F_{n,i,k}^{2}} - v_{i,k} } \xrightarrow[n \to \infty]{} 0.
\end{equation}

\begin{lemma}
  The vector
  \begin{equation*}
    \paren*{ A_{n,i,k,2}, A_{n,i,k,1}, R_{n,i,l_{n}}, X_{n,i,k} : i \in \mathbb{N}, k \in \mathbb{N}^{*} },
  \end{equation*}
  is $1$-admissible.
\end{lemma}

\begin{proof}
  By construction $A_{n,i,k,2} \in \mathcal{W}_{0}$, and $A_{n,i,k,1}, X_{n,i,k} \in \mathcal{W}_{1}$.
  Let us check that $R_{n,i,l_{n}}$ is asymptotically Gaussian.
  By definition
  \begin{equation*}
    \rho_{1}(R_{n,i,l_{n}}) \leq \norm{\Gamma(R_{n,i,l_{n}}, X_{n,i,l_{n}+1})} + \frac{1}{n}.
  \end{equation*}
  As explained previously, decomposing $R_{n,i,l_{n}} = F_{n,i,l_{n}+1} + R_{n,i,l_{n}+1}$, we find that $R_{n,i,l_{n}+1}$, $A_{n,i,l_{n}+1,2}$ and $A_{n,i,l_{n}+1,2}$ are independent of $X_{n,i,l_{n}+1}$ and that
  \begin{equation*}
    \begin{split}
      \Gamma(R_{n,i,l_{n}}, X_{n,i,l_{n}+1}) &= A_{n,i,l_{n}+1,2} \Gamma(H_{2}(X_{n,i,l_{n}+1}), X_{n,i,l_{n}+1}) + A_{n,i,l_{n}+1,1} \Gamma(X_{n,i,l_{n}+1}, X_{n,i,l_{n}+1}).
                                           \\&= 2A_{n,i,l_{n}+1,2} + A_{n,i,l_{n+1},1}.
    \end{split}
  \end{equation*}
  By independence and since $\Esp*{X H_{2}(X)} = 0$ for $X$ Gaussian, we have
  \begin{equation*}
    \norm{F_{n,i,l_{n}+1}}^{2} = \norm{A_{n,i,l_{n}+1,2} H_{2}(X_{n,i,l_{n}+1}) + A_{n,i,l_{n}+1} X_{n,i,l_{n}+1}}^{2} = 4 A_{n,i,l_{n}+1,2}^{2} + \norm{A_{n,i,l_{n}+1,1}^{2} }.
  \end{equation*}
  Combining with \cref{eq:intro:convergence-max}, we have shown that
  \begin{equation*}
    \rho_{1}(R_{n,i,l_{n}+1}) \leq 4 \norm{F_{n,i,l_{n}+1}} + \frac{1}{n} \leq c v^{1/2}_{i,l_{n}+1} + \frac{1}{n}+o(1)
  \end{equation*}
  Since all the terms in \cref{eq:intro:decomposition-F} are orthogonal, and $F_{n}$ is normalized, we also find that
  \begin{equation*}
    \sum_{k} v_{i,k} \leq 1.
  \end{equation*}
  Thus $v_{i,l_{n}+1} \to 0$.
  This shows that $\rho_{1}(R_{n,i,l_{n}+1}) \to 0$ as $n \to \infty$, and by \cref{th:wiener:normal-convergence} we conclude that it is asymptotically Gaussian.
\end{proof}

By the induction hypothesis, we can thus, up to extraction, find $A_{\infty,i,k,2} \in \mathcal{W}_{0}$, $A_{\infty,i,k,1} \in \mathcal{W}_{\leq 1}$, $R_{\infty,i} \in \mathcal{W}_{\leq 1}$, and $X_{\infty,i,k} \in \mathcal{W}_{\leq 1}$ such that
\begin{equation*}
\paren*{ A_{n,i,k,2}, A_{n,i,k,1}, R_{n,i,l_{n}}, X_{n,i,k} : i \in \mathbb{N}, k \in \mathbb{N}^{*} } \xrightarrow[n \to \infty]{\law} \paren*{ A_{\infty,i,k,2}, A_{\infty,i,k,1}, R_{\infty,i}, X_{\infty,i,k} : i \in \mathbb{N}, k \in \mathbb{N}^{*} }
\end{equation*}
By our inductive construction where $F_{n,i,k+1}$ is independent of $(X_{n,i,1}, \dots, X_{n,i,k})$ and since $F_{n,i}$ is normalised, we find that
\begin{equation*}
  1 \geq \norm*{ \sum_{k} F_{n,i,k} }^{2}_{L^{2}(\prob)} = \sum_{k} \Esp*{F_{n,i,k}^{2}}.
\end{equation*}
Since we work with Gaussian polynomials the orthogonality is preserved by the convergence in law, see \cref{th:polynomials:convergence-moment}, thus using that the norm is weakly lower semi-continuous, we get
\begin{equation*}
  1 \geq \norm*{ \sum_{k} F_{\infty,i,k} }^{2}_{L^{2}(\prob)} = \sum_{k} \Esp*{F_{\infty,i,k}^{2}}.
\end{equation*}
This shows that the series $\sum_{k} F_{\infty,i,k}$ is convergent in $L^{2}(\prob)$ and is an element of $\mathcal{W}_{\leq 2}$ since this space is closed in $L^{2}(\prob)$.
Thus let us define
\begin{equation*}
  F_{\infty,i} \coloneq \sum_{k} F_{\infty,i,k} + R_{\infty,i} \in \mathcal{W}_{\leq 2}.
\end{equation*}
We conclude the proof by showing the following convergence.
\begin{lemma}
  The sequence $(\vec{F}_{n})$ converges in law to $\vec{F}_{\infty} \coloneq (F_{\infty,i})$.
\end{lemma}
\begin{proof}
Take a finite subset $J=\{j_1,\cdots,j_p\}\subset \mathbb{N}$. Then, since $l_n\to\infty$, there exists $N_0$ such that for any $n\ge N_0$ we have $J\subset \{0,\cdots,l_n\}$. In virtue of \cref{eq:intro:convergence-max} we have in particular
\[\sum_{i=1}^{l_n}\sum_{k=1}^{l_n} \left|\E\left[F_{n,i,k}^2\right]-v_{i,k}\right|\to 0.\]
Let $\epsilon>0$, one may find $K_\epsilon\ge 1$ and $N_1\ge N_0$ such that for any $n\ge N_1$ we have
\[
\sum_{l=1}^p\sum_{k={K_\epsilon+1}}^{l_n} \E\left[F_{n,j_l,k}^2\right]+\underbrace{\E\left[F_{\infty,j_l,k}^2\right]}_{=v_{j_l,k}}\le \epsilon.
\]
Thus, with $\mathbf{W}_{2}$ the $2$-Wasserstein defined in \cref{eq:def-wasserstein}, one recovers
\begin{eqnarray*}
&&\mathbf{W}_2\left(\left(F_{n,j}\right)_{j \in J}, \left(F_{\infty,j}\right)_{j \in J}\right) \\
&&\leq \epsilon + \mathbf{W}_2\left(\left(\sum_{k=1}^{K_\epsilon} F_{n,j_l,k} + R_{n,j_l,l_n}\right)_{j \in J}, \left(\sum_{k=1}^{K_\epsilon} F_{\infty,j_l,k} + R_{\infty,j_l}\right)_{j \in J}\right).
\end{eqnarray*}
Letting $n\to \infty$ and using induction hypothesis implies that \[\mathbf{W}_2\left(\left(\sum_{k=1}^{K_\epsilon} F_{n,j_l,k} + R_{n,j_l,l_n}\right)_{j \in J}, \left(\sum_{k=1}^{K_\epsilon} F_{\infty,j_l,k} + R_{\infty,j_l}\right)_{j \in J}\right)\to 0.\]
Letting $\epsilon\to 0$ completes the proof.
\end{proof}
\subsubsection{The quadratic case implies the cubic case}

The mechanism of proof is identical, we consider $\left(\vec{F_n}\right)$ an admissible sequence of order $3$.
As before we focus on the set $I$ of indices such that $(F_{n,i}); i\in I$ does not have a Gaussian limit. Then, we notice that, as in the case of degree $2$, for Wiener chaoses of degree $3$, the semi-norm $\rho_1$ controls the asymptotic Gaussianity. Thanks to this observation, we may perform similar decompositions: for $R_{n,i,k}$ being constructed, we set
\[R_{n,i,k}=A_{n,i,k+1,3} H_3(X_{n,i,k+1})+A_{n,i,k+1,2} H_2(X_{n,i,k+1})+A_{n,i,k+1,1} H_1(X_{n,i,k+1})+R_{n,i,k+1}.\]
There, $A_{n,i,k+1,3}$ is in $\mathcal{W}_0$ hence constant, $A_{n,i,k+1,2}$ is in $\mathcal{W}_1$, and $A_{n,i,k+1,1}$ is in $\mathcal{W}_2$. All are independent of $X_{n,i,k+1}$ which is in $\mathcal{W}_1$ and satisfies $\rho_{1}(R_{n,i,k})\le \|\Gamma\left[R_{n,i,k},X_{n,i,k+1}\right]\|+\frac 1 n$. Setting $F_{n,i,k}\coloneq A_{n,i,k+1,3} H_3(X_{n,i,k+1})+A_{n,i,k+1,2} H_2(X_{n,i,k+1})+A_{n,i,k+1,1} H_1(X_{n,i,k+1})$ one can also write (with the exact same orthogonality and independence as in the quadratic case) for some $r_n\to\infty$
\[ F_{n,i}=\sum_{k=1}^{r_n} F_{n,i,k}+R_{n,i,r_n};\,\text{where}\,\,R_{n,i,r_n}\,\,\text{has Gaussian limit}.\]
Then, in order to conclude the proof, one is left to apply the induction hypothesis to the $2$-admissible sequence $(A_{n,i,k,3},A_{n,i,k,2},A_{n,i,k,1},X_{n,i,k},R_{n,i,l_n}: i,k\ge 1)$ and to truncate the series tails $\sum_{k=K_{\epsilon}+1} ^ {l_n} F_{n,i,k}$ in the exact same manner.
This truncation procedure relies entirely on the fact that $(F_{n,i,k})_{k\ge 1}$ is an orthogonal sequence and this still holds for degree $3$ Wiener chaoses.
For degree strictly higher than $3$, this orthogonality must be replaced by a weaker asymptotic orthogonality but we do not discuss this in the outline.

\section{Preliminaries}
Since we are working with polynomials with random inputs, let us gather some facts about those.

We start by recalling a direct consequence of hypercontractivity for polynomials \cite[\S~3.2]{MOO}.
\begin{theorem}\label{th:polynomials:convergence-moment}
  Let $(P_{n})$ be a sequence of multivariate polynomials of at most $d \in \mathbb{N}$.
  Let $\vec{X} = (X_{i})$ be a sequence of independent and identically distributed random variables, that are centered, with unit variance, and such that $\Esp{\abs{X_{1}}^{p}} < \infty$ for all $p \in \mathbb{N}$.
  Then, if $(P_{n}(\vec{X}))$ converges in law to some $Y_{\infty}$, then
  \begin{equation*}
    \Esp*{ \abs*{ P_{n}(\vec{X}) }^{p} } \xrightarrow[n \to \infty]{} \Esp*{ \abs*{ Y_{\infty} }^{p}}, \qquad p \in \mathbb{N}.
  \end{equation*}
\end{theorem}

The following is also standard and follows, for instance, from \cite[Prop.~3.5]{MOO}.
\begin{theorem}\label{th:polynomials:closed-L2}
  Let $\vec{X}$ be as above.
  Then the set
  \begin{equation*}
    \set*{ P(\vec{X}) : P \ \text{multivariate polynomial with degree at most $d$ }},
  \end{equation*}
is closed in $L^{2}(\prob)$.
\end{theorem}

Let also recall the definition if the Wasserstein distance.
Here $d \in \mathbb{N}^{*}$, and $X$ and $Y$ are random vectors in $\mathbb{R}^{d}$.
\begin{equation}\label{eq:def-wasserstein}
  \mathbf{W}_{2}(X,Y) \coloneq \Inf*{ \bracket*{\Esp*{ \abs{\tilde{X} - \tilde{Y}}^{2} }}^{1/2} : \Law{\tilde{X}} = \Law{X},\, \Law{\tilde{Y}} = \Law{Y} },
\end{equation}
where $\abs{\cdot}$ is the Euclidean norm on $\mathbb{R}^{d}$.
Following \cite[Thm.~6.9]{Villani}, the Wasserstein distance metrizes the topology of the convergence in law together with the convergence of the second moment.
Combining with \cref{th:polynomials:convergence-moment}, we obtain the following.

\begin{corollary}\label{th:wasserstein-polynomial}
  For all $d \in \mathbb{N}$ and $l \in \mathbb{N}^{*}$, on the set
  \begin{equation*}
    \set*{ (P_{1}(\vec{X}), \dots, P_{l}(\vec{X})) : (P_{i})_{i} \ \text{multivariate polynomial with degree at most $d$ }},
  \end{equation*}
the topology of the convergence in law is equivalent to that induced by the Wasserstein distance.
\end{corollary}

We finish this short section with an elementary lemma that we use several times.
\begin{lemma}\label{th:analytical-lemma}
Consider a sequence of sequences $(u_{n,m})_{n,m\ge 1}$ such that for any $m\ge 1$, $u_{n,m}\xrightarrow[n\to\infty]~0$.
Then, there exists a real sequence $(r_{n})$ such that $r_n\to\infty$ and
\begin{equation*}
\sum_{k=1}^{r_n}\left| u_{n,k}\right|\xrightarrow[n\to\infty]~0.
\end{equation*}
\end{lemma}
\begin{proof}
One may build an increasing sequence of integers $(N_i)_{i\ge 1}$ such that for any $n\ge N_i$ we have
\[\sum_{k=1}^i |u_{n,k}|\le \frac{1}{i}.\]
Assume that we have $(N_1,\cdots,N_{i})$ fulfilling the above conditions. There exists $N_{i+1}> N_i$ such that for any $k\in\llbracket 1,i\rrbracket$ and $n\ge N_{i+1}$ we have $|u_{n,k}|\le \frac{1}{(i+1)^2}$. This entails that for any $n\ge N_{i+1}$, $\sum_{k=1}^{i+1}|u_{n,k}|\le \frac{1}{i+1}.$ Then, for any $n\in\llbracket N_i, N_{i+1}\llbracket$ one sets $r_n=i$. By construction, on this integers interval one has $\sum_{k=1}^{r_n} |u_{n,k}|\le \frac{1}{i}$ which concludes the proof.
\end{proof}

\section{Stability of Wiener chaoses}\label{s:proof-wiener}
We prove the following infinite-variate version of \cref{th:stabilite-wiener}.
\begin{theorem}\label{th:stabilite-wiener-multidim}
  Let $(P_{i,n})_{i,n}$ be polynomials of degree at most $d \in \mathbb{N}^{*}$; let $\vec{G}$ be a standard Gaussian vector; and let $\vec{F}_{n}$ be the infinite random vector given by
  \begin{equation*}
    F_{i,n} \coloneq P_{i,n}(\vec{G}), \qquad i,n \in \mathbb{N}.
  \end{equation*}
  Assume that $(\vec{F}_{n})$ converges in law.
  Then, there exist polynomials $(Q_{i})_{i}$ of degree at most $d$ such that
  \begin{equation*}
    \vec{F}_{n} \xrightarrow[n \to \infty]{\law} (Q_{i}(\vec{G}))_{i}.
  \end{equation*}
\end{theorem}

\subsection{Reminders on the Wiener space}\label{s:wiener:reminder}
In this section, we review the necessary materials, and only the necessary materials, regarding Wiener chaoses.
Readers interested in a broader introduction or the most general definitions should read \cite{BouleauHirsch,Janson,Nualart,NourdinPeccatiBlueBook} for thorough introduction on the subjects.
\subsubsection{The Wiener space}
In this section, we work on the \emph{Wiener space} which is the probability space
\begin{equation*}
  (\Omega, \mathfrak{W}, \prob) \coloneq (\mathbb{R}, \mathfrak{B}(\mathbb{R}), \gamma)^{\otimes \infty},
\end{equation*}
where $\mathfrak{B}(\mathbb{R})$ is the Borel $\sigma$-algebra of $\mathbb{R}$, and $\gamma = \mathcal{N}(0,1)$ is the standard Gaussian distribution.
We equip it with the canonical coordinates process
\begin{equation*}
  G_{i}(\omega) \coloneq \omega_{i}, \qquad \omega \in \Omega, \qquad i \in \mathbb{N},
\end{equation*}
in a way that the vector $\vec{G} = (G_{i})$ is an infinite-vector of independent standard Gaussian variables.
The Wiener space is invariant under orthogonal transformation, that is if $\mathsf{A} \in \mathscr{O}(\ell^{2})$ then $\mathsf{A} \vec{G}$ has the same law as $\vec{G}$, and all the definitions below are also independent of the choice of the basis.

\subsubsection{Wiener chaoses}
Recall that we have defined $\mathcal{W}_{d}$, the Wiener chaos of degree $d$ in \cref{eq:def-wiener-chaos}, as well as $\mathcal{W}_{\leq d}$, the sum of chaos of degree at most $d$, which coincides with the space of Gaussian polynomials of degree at most $d$.
From the definition of $\mathcal{W}_{\leq d}$ in terms of polynomials, we see that
\begin{equation}\label{eq:wiener:multiplication}
  \bracket*{F \in \mathcal{W}_{\leq m}, X \in \mathcal{W}_{\leq d}} \Longrightarrow FX \in \mathcal{W}_{\leq (d + m)}.
\end{equation}
For $F \in \mathcal{W}_{\leq d}$, we write $\mathsf{J}_{m}F$ for its orthogonal projection onto $\mathcal{W}_{m}$.
The Hermite polynomials form an orthonormal basis of $L^{2}(\gamma)$, from which we see that
\begin{equation*}
  \mathcal{W}_{\leq d} \subset \bigcap_{p < \infty} L^{p}(\prob).
\end{equation*}
Actually on Wiener chaoses, all the $L^{p}$-norms are equivalent.
Namely, following \cite[\S 3.1.3]{HMPSuperConvergence}, for $d \in \mathbb{N}$, and $1 \leq p \leq q < 1$, there exists $c = c_{m,p,q}$ such that
\begin{equation}\label{eq:wiener:norm-equivalence}
  \norm{F}_{q} \leq c \norm{F}_{p}, \qquad F \in \mathcal{W}_{\leq d}.
\end{equation}

From this we obtain the following characterization.
\begin{lemma}\label{th:wiener:convergence-in-law-moments}
  Let $K \subset \mathcal{W}_{\leq d}$.
  The following are equivalent.
  \begin{enumerate}[(i)]
    \item $K$ is bounded in some $L^{p}(\prob)$ ($p < \infty$).
    \item $K$ is relatively compact for the topology of the convergence in law and the convergence of all moments.
    \item $K$ is relatively compact for the topology of the Wasserstein distance.
  \end{enumerate}
\end{lemma}

\begin{remark}
  Our \cref{th:stabilite-wiener} actually completes this result by stating that $L^{p}$ balls in $\mathcal{W}_{\leq d}$ are actually compact for the topology of the Wasserstein distance.
\end{remark}

\subsubsection{Differential operators on the Wiener spaces}
Let us define the \emph{Ornstein--Uhlenbeck generator}
\begin{equation*}
  \mathsf{L} F \coloneq \sum_{m=0}^{d} (-d) \mathsf{J}_{m} F, \qquad F \in \mathcal{W}_{\leq d}.
\end{equation*}
The operator $\mathsf{L} \colon \mathcal{W}_{\leq d} \to \mathcal{W}_{\leq d}$ is continuous for the $L^{2}(\prob)$-topology, and its eigenspaces are exactly $\mathcal{W}_{m}$.
Related, to $\mathsf{L}$, we define the \emph{carré du champ} operator
\begin{equation*}
  \Gamma[F,X] \coloneq \frac{1}{2} \paren*{ \mathsf{L}(FX) - F\mathsf{L}X - X\mathsf{L}F }, \qquad F,X \in \mathcal{W}_{\leq d}.
\end{equation*}
For Wiener chaoses, the carré du champ has a particularly simple form
\begin{equation}\label{eq:wiener:gamma-chaos}
  \Gamma(F, X) = \frac{1}{2} (\mathsf{L} + p + q)(FX), \qquad F \in \mathcal{W}_{p}, X \in \mathcal{W}_{q}.
\end{equation}
Similarly to \cref{eq:wiener:multiplication}, the carré du champ satisfies a multiplication property
\begin{equation}\label{eq:wiener:multiplication-gamma}
  \bracket*{ F \in \mathcal{W}_{\leq d}, G \in \mathcal{W}_{\leq m} } \Longrightarrow \Gamma[F,G] \in \mathcal{W}_{\leq (d+m-2)}.
\end{equation}
\subsubsection{Smooth random variables}
We define the space $\mathbb{D}^{\infty}$ as the closure of $\cup_{m \in \mathbb{N}} \mathcal{W}_{\leq m}$ for the family of seminorms
\begin{equation*}
  F \mapsto \norm{\mathsf{L}^{k} F}_{L^{p}(\prob)}, \qquad k \in \mathbb{N}, p \in (1,\infty).
\end{equation*}
It is known that $\mathbb{D}^{\infty}$ is an algebra stable under $\mathsf{L}$, $\Gamma$, and composition with smooth functions with polynomial growth.
This allows to state the important \emph{chain rule formula} for $\Gamma$
\begin{equation}\label{eq:wiener:chain-rule}
  \Gamma[\varphi(F), G] = \varphi'(F) \Gamma[F,G], \qquad F,G \in \mathbb{D}^{\infty}, \varphi \in \mathscr{C}^{\infty}_{pol}.
\end{equation}
In particular, the carré du champ can be understood as the square $\ell^{2}$-norm of the infinite-dimensional gradient
\begin{equation}\label{eq:wiener:gamma-ell-2}
  \Gamma(\varphi(\vec{G}), \psi(\vec{G})) = \sum_{i \in \mathbb{N}} \partial_{i}\varphi(\vec{G}) \partial_{i} \psi(\vec{G}), \qquad \varphi,\psi \in \mathscr{C}^{\infty}_{pol}.
\end{equation}
Finally, the following \emph{integration by parts formula} plays a crucial role in our analysis
\begin{equation}\label{eq:wiener:ipp}
  - \Esp*{F \mathsf{L}G} = \Esp*{ \Gamma[F,G] }, \qquad F,G \in \mathscr{C}^{\infty}_{pol}.
\end{equation}

An immediate consequence of the definition and the properties is the following result that we use repeatedly.
\begin{lemma}\label{th:algebraic-lemma}
  Let $(F,G,H)\in\mathcal{W}_p\times \mathcal{W}_q\times\mathcal{W}_r$ with $(p,q,r)\in\mathbb{N}^3$.
  Then,
  \begin{equation}\label{eq-lem-alg}
    \Esp*{ \Gamma[F,G] H} = \frac{p+q-r}{2} \Esp*{ F G H}.
  \end{equation}
\end{lemma}

\begin{proof}
  By definition of $\Gamma$ and the fact that $F$ and $G$ are chaotic, we find
  \begin{equation*}
    \Gamma[F,G] = \frac{1}{2} \paren*{ \mathsf{L} + p+q }(FG).
  \end{equation*}
  By \cref{eq:wiener:ipp},
  \begin{equation*}
    \Esp*{ H \mathsf{L}(FG) } = \Esp*{ FG \mathsf{L} H } = - r \Esp*{ FGH }.
  \end{equation*}
\end{proof}

We also need the following inequality.
\begin{lemma}\label{th:wiener:spectral-inequality}
  Let $X \in \mathcal{W}_{p}$ and $Y \in \mathcal{W}_{q}$.
  Then
  \begin{equation*}
     \Esp*{ \Gamma(X,Y)^{2} } \leq \frac{p+q}{4} \Esp*{ XY \Gamma(X,Y) }.
  \end{equation*}
\end{lemma}

\begin{proof}
  Using \cref{eq:wiener:gamma-chaos,eq:wiener:ipp}, we find that
  \begin{equation*}
  \Esp*{ \Gamma(X,Y)^{2} } = \frac{1}{4} \Esp*{ (XY) \mathsf{L} (\mathsf{L} + p + q)(XY) } + \frac{1}{4} (p+q) \Esp*{(XY) \Gamma(X,Y) }.
\end{equation*}
By definition, $\mathsf{L}$ is a non-positive operator, and, on the other, by \cref{eq:wiener:multiplication}, $XY \in \mathcal{W}_{\leq (p+q)}$ on which $\mathsf{L} + p + q$ is a non-negative operator.
Thus, the first term on the right-hand side is non-positive, and this yields the inequality.
\end{proof}

\subsection{Directional influences: a new criterion for asymptotic normality on Wiener chaoses}\label{s:influences}

We introduce the \emph{directional influences of degree $k$}, we new tool we develop to study convergence in law for a sequence of random variables in a Wiener chaos of fixed degree.
We use them to formulate a new necessary and sufficient condition for such a sequence to have a Gaussian limit.
This criterion, of independent interest, plays a prominent role in our approach.

\begin{definition}\label{df:directional-influence}
  Let $q \in \mathbb{N}^{*}$, the \emph{directional influence of degree $q$} is defined as
  \begin{equation}\label{eq:directional-influence}
    \rho_{q}(F) = \Sup*{ \norm{ \Gamma(F,X) : X \in \mathcal{W}_{q}, \Esp*{X^{2}} = 1 } }, \qquad F \in \mathbb{D}^{\infty}.
  \end{equation}
\end{definition}

Since we use this construction often let us also define.
\begin{definition}
  Let $(F_{n})$ be a sequence in $\mathbb{D}^{\infty}$ and $q \in \mathbb{N}^{*}$.
  We say that a sequence $(X_{n})$ \emph{realises} $\rho_{q}(F_{n})$ provided $X_{n} \in \mathcal{W}_{q}$ with $\Esp*{ X_{n}^{2} } = 1$ is such that
  \begin{equation*}
    \norm{ \Gamma[F_{n}, X_{n} ] }_{L^{2}(\prob)} \ge \rho_{q}(F_{n}) - \frac{1}{n}.
  \end{equation*}
\end{definition}
The following monotonicity property is immediate.

\begin{lemma}\label{th:directional-influence-monotone}
  For $1\le q<p$ we have $\rho_q\le \rho_p$.
\end{lemma}

\begin{proof}
  Take $q<p$, $F \in \mathbb{D}^{\infty}$ and $(X_{n})$ realizing $\rho_{q}(F)$.
  We first assume that $F$ depends on finitely many Gaussian coordinates say $G_{1}, \dots, G_{m}$ for some $m \in \mathbb{N}^{*}$.
  In this case we can also take $X_{n}$ only depending on $G_{1}, \dots, G_{m}$.
  Can consider $(Y_n)_{n\ge 1}$ in $\mathcal{W}_{p-q}$, only depending on $G_{m+1}, G_{m+2}, \dots$ with $\Esp*{ Y_{n}^{2} } = 1$.
  In particular, $Y_{n}$ is independent of $(F,X_n)_{n\ge 1}$ and $\Gamma(F, Y_{n}) = \Gamma(X_{n}, Y_{n}) = 0$.
  Thus,
\begin{equation}\label{eq:monotone-limit}
  \Esp*{\Gamma[F, X_n Y_n]^2} = \Esp*{Y_n^2 \Gamma[F,X_n]^2} = \norm{\Gamma[F,X_n]}_{L^{2}(\prob)}^{2} \xrightarrow[n \to \infty]{} \rho_{q}(F).
\end{equation}
Using that the carré du champ vanished, we also find that
\begin{equation*}
  \mathsf{L}(X_{n}Y_{n}) = X_{n} \mathsf{L} Y_{n} + Y_{n} \mathsf{L} X_{n} = -p X_{n} Y_{n}.
\end{equation*}
Thus, $X_n Y_n \in\mathcal{W}_{p}$, hence $\norm{\Gamma[F,X_nY_n]}_2\le \rho_p(F)$ by definition.
Combining with \cref{eq:monotone-limit} and taking $n \to \infty$, we conclude.
For the general case, we approximate $F$ by a sequence $(F_{n})$ in $\mathbb{D}^{\infty}$ depending only on the $m_{n}$ first coordinates.
\end{proof}

\subsubsection{Revisiting asymptotic normality on Wiener chaoses}
The following result summarizes some known criterion for normal convergence, and establish the equivalence with our new criterion based on directional influence.
To the best of our knowledge, it is new.

\begin{theorem}\label{th:wiener:normal-convergence}
  Let $(F_{n})$ be a sequence in $\mathcal{W}_{p}$.
  The following are equivalent.
  \begin{enumerate}[(i)]
    \item\label{i:wiener:normal-convergence} $(F_{n})$ is asymptotically Gaussian.
    \item\label{i:wiener:normal-convergence:variance} $\var \Gamma[F_{n}, F_{n}]$ converges to $0$.
    \item\label{i:wiener:normal-convergence:eigenvalue} for $k \geq 2$, $H_{k}(F_{n})$ is asymptotically an eigenvalue for $\mathsf{L}$ of order $kp$
      \begin{equation*}
        (\mathsf{L} + kp) H_{k}(F_{n}) \xrightarrow[n \to \infty]{L^{2}(\prob)} 0.
      \end{equation*}
      \item\label{i:wiener:normal-convergence:projection} for $k \geq 2$, $H_{k}(F_{n})$ has asymptotically no components other than $kp$
      \begin{equation*}
        \mathcal{J}_{q} H_{k}(F_{n}) \xrightarrow[n \to \infty]{L^{2}(\prob)} 0, \qquad q \ne kp.
      \end{equation*}
    \item\label{i:wiener:normal-convergence:influence-max} $\rho_{p-1}(F_{n})$ vanishes as $n \to \infty$. 
    \item\label{i:wiener:normal-convergence:influence-min} $\rho_{\floor*{\frac{p}{2}}}(F_{n})$ vanishes as $n \to \infty$. 
  \end{enumerate}
\end{theorem}

\begin{proof}
  \cref{i:wiener:normal-convergence} $\Leftrightarrow$ \cref{i:wiener:normal-convergence:variance} follows from \cite{NualartOrtizLatorre}.

  \cref{i:wiener:normal-convergence:variance} $\Leftrightarrow$ \cref{i:wiener:normal-convergence:eigenvalue} is immediate once we observe that $(\mathsf{L} + kp) H_{k}(F_{n}) = k(k-1) H_{k-2}(F_{n}) (\Gamma(F_{n}, F_{n}) - p)$, see \cite[Eq.~(4.6)]{AMMP}.

  \cref{i:wiener:normal-convergence:eigenvalue} $\Leftrightarrow$ \cref{i:wiener:normal-convergence:projection} since $(\mathsf{L} + 2p)$ corresponds exactly to projecting on $\mathcal{W}_{2p}^{\perp}$, up to a multiplicative constant.

  We can assume that $p>1$, otherwise, on the one hand $(F_{n})$ is already Gaussian, and on the other hand $\mathcal{W}_{p-1} = \mathcal{W}_0 = \mathbb{R}$.
  Thus, the chain rule \cref{eq:wiener:chain-rule} implies that $\Gamma[F_n,X]=0$ for every $X\in\mathcal{W}_0$, hence $\rho_{\floor{\frac{p}{2}}}(F_{n}) = \rho_{p-1}(F_n)=0$.
      Thus, we assume that $p\ge 2$.

  \cref{i:wiener:normal-convergence:projection,i:wiener:normal-convergence:eigenvalue} $\Rightarrow$ \cref{i:wiener:normal-convergence:influence-max}.
      Consider $(F_n)_{n\ge 1}\in\mathcal{W}_p$ such that $F_n\to\mathcal{N}(0,1)$ and let us prove that $\rho_{p-1}(F_n)\to 0$.
      Take $(X_{n})$ realising $\rho_{p-1}(F_{n})$.
      Since $X_n\in\mathcal{W}_{p-1}$, by \cref{eq:wiener:multiplication}, we find $X_n^2 \in \mathcal{W}_{\leq (2p-2)}$.
      Given that $\mathcal{W}_{2p} \perp \mathcal{W}_{\leq (2p-2)}$, we derive, by \cref{i:wiener:normal-convergence:projection}, that $\Esp*{H_{2}(F_{n}) X_n^2} \to 0$.
      On the other hand by \cref{i:wiener:normal-convergence:eigenvalue} $\Esp*{ (\mathsf{L} + 2p)H_{2}(F_{n}) X_{n}^{2}} \to 0$.
      Using the chain rule \cref{eq:wiener:chain-rule}, the integration by parts \cref{eq:wiener:ipp}, we thus find
      \begin{equation*}
        \begin{split}
          \Esp*{F_n X_n \Gamma[F_n,X_n]}
          &= \frac{1}{4}\Esp*{\Gamma[X_n^2,H_{2}(F_{n})]}
        \\&= -\frac{1}{4}\Esp*{ X_n^2 \mathsf{L} H_{2}(F_{n}) }
        \\&= -\frac{1}{4}\Esp*{ X_n^2 (\mathsf{L} +2p) H_{2}(F_{n}) } + \frac{p}{2} \Esp*{ X_{n}^{2} H_{2}(F_{n}) }
          \xrightarrow[n \to \infty]{} 0.
        \end{split}
      \end{equation*}
      By \cref{th:wiener:spectral-inequality}, we conclude that
      \begin{equation*}
        \Gamma(F_{n}, X_{n}) \xrightarrow[n \to \infty]{L^{2}(\prob)} 0,
      \end{equation*}
      and thus $\rho_{p-1}(F_{n}) \to 0$.

      \cref{i:wiener:normal-convergence:influence-max} $\Rightarrow$ \cref{i:wiener:normal-convergence:influence-min}. Direct consequence of the monotonicity \cref{th:directional-influence-monotone}.

      \cref{i:wiener:normal-convergence:influence-min} $\Rightarrow$ \cref{i:wiener:normal-convergence:variance}.
      We want to establish that $\Var*{ \Gamma(F_{n}, F_{n}) } \to 0$ as $n \to \infty$.
      By \cref{eq:wiener:multiplication-gamma}, it is sufficient to show that for all $(Z_n)_{n\ge 1}$ a bounded sequence in $\mathcal{W}_k$ with $k\in \set*{ 1, \dots, 2p-2}$. $\Esp*{\Gamma[F_n,F_n]Z_n}\to 0$.
      Indeed, this yields that orthogonal projections of $\Gamma[F_n,F_n]$ on chaoses of order $k$ tend to zero, hence the desired claim.
      For every $l \leq p \wedge k$, and every multi-index $(i_1,\cdots,i_l)\in\mathbb{N}^l$ we have $\partial_{i_1,\cdots,i_l} F_n\in\mathcal{W}_{p-l}$ and $\partial_{i_1,\cdots,i_l} Z_n\in\mathcal{W}_{k-l}$.
      Provided that $(p-l)+(k-l)\neq p$, that is $k\neq 2l$, recalling \cref{eq:wiener:gamma-ell-2}, \cref{th:algebraic-lemma} entails that
      \begin{equation*}
        \begin{split}
          \Esp*{F_n \partial_{i_1,\cdots,i_l} F_n \partial_{i_1,\cdots,i_l} Z_n}
&=\frac{2}{k-2l}\Esp*{F_n \Gamma[\partial_{i_1,\cdots,i_l} F_n ,\partial_{i_1,\cdots,i_l} Z_n]}
\\&=\frac{2}{k-2l}\sum_{i_{l+1}=1}^\infty\Esp*{F_n \partial_{i_1,\cdots,i_{l+1}} X_n\partial_{i_1,\cdots,i_{l+1}} Z_n}.
        \end{split}
      \end{equation*}
      Applying $r$-times consecutively the previous procedure leads to
      \begin{equation}\label{eq:directional:algebraic-iteration}
        \begin{split}
          \Esp*{F_n F_n Z_n} &= \frac{2}{k}\Esp*{F_n\Gamma[F_n,Z_n]}
                           \\&= \frac{2}{k}\sum_{i=1}^\infty \Esp*{F_n \partial_i F_n \partial_i Z_n}
                           \\&\vdots
                           \\&= \frac{2^r}{\prod_{l=0}^{r-1}(k-2l)}\sum_{i_1,\cdots,i_r=1}^\infty\Esp*{F_n \partial_{i_1,\cdots,i_r} F_n\partial_{i_1,\cdots,i_r} Z_n}.
        \end{split}
      \end{equation}
      We can apply this procedure as long as $k\neq 2l$ for every $l\in \{ 0, \dots, r-1 \}$.
      Let $r_0$ be maximal such that the above equality holds, in particular $p=(p-r_0)+(k-r_0)$, otherwise one could further decompose contradicting the maximality of $r_0$.
      Then, either $p-r_0\le \frac{p}{2}$ or $k-r_0\le \frac{p}{2}$.
      Otherwise $p-r_0>\frac{p}{2}$ and $k-r_0>\frac{p}{2}$ which contradicts $p=(p-r_0)+(k-r_0)$.

      \begin{enumerate}[(a),wide]
        \item \textbf{Case $p-r_0\le \frac{p}{2}$}.
      In this case, $p-r_0\le \floor*{ \frac{p}{2}}$, since we are working with integers.
      By \cref{th:algebraic-lemma}, for a given multi-index $(i_1,\cdots,i_{r_0})$ and given that $p+(p-r_0)=2p-r_0>k-r_0$ ($k\in \{ 1, \dots, 2p-2 \}$) we obtain
      \begin{equation}\label{eq:directional:algebraic-reverse}
        \Esp*{F_n \partial_{i_1,\cdots,i_{r_0}} F_n\partial_{i_1,\cdots,i_{r_0}} Z_n} = \frac{2}{2p-k}\Esp*{\Gamma[F_n,\partial_{i_1,\cdots,i_{r_0}} F_n]\partial_{i_1,\cdots,i_{r_0}} Z_n}.
      \end{equation}
      Let us write $m \coloneq \floor*{\frac{p}{2}}$.
      Besides, for any $X\in\mathcal{W}_{p-r_0}$ we have $\|\Gamma[F_n,X]\|_2\le \|X\|_2 \rho_{p-r_0}(F_n)\le \|X\|_2 \rho_{m}(F_n)$ since $p-r_0 \le m$ and since we have \cref{th:directional-influence-monotone}.
      Hence, using Cauchy-Schwarz with $\Gamma[F_n,\partial_{i_1,\cdots,i_{r_0}} F_n]$ and $\partial_{i_1,\cdots,i_{r_0}} Z_n$ we have,
      \begin{equation*}
        \norm*{F_n \partial_{i_1,\cdots,i_{r_0}} F_n\partial_{i_1,\cdots,i_{r_0}} Z_n}_{L^{1}(\prob)} \le\frac{2\rho_{m}(F_n)}{2p-k} \norm*{\partial_{i_1,\cdots,i_{r_0}} F_n}_{L^{2}(\prob)} \norm*{\partial_{i_1,\cdots,i_{r_0}} Z_n }_{L^{2}(\prob)}.
      \end{equation*}
      Gathering these facts, we may write
      \begin{equation}
        \begin{split}
          \abs*{ \Esp*{\Gamma[F_n,F_n]Z_n} } & = \tfrac{2p-k}{2} \abs*{ \Esp*{F_n^2 Z_n} } \qquad (\text{\cref{th:algebraic-lemma}}).
                                           \\&\leq \tfrac{2p-k}{2}\times \tfrac{2^{r_0}}{\prod_{l=0}^{r_0-1} \Esp*{k-2l}}  \sum_{i_1,\cdots,i_{r_0}=1}^\infty \abs*{\Esp*{F_n \partial_{i_1,\cdots,i_{r_0}} F_n\partial_{i_1,\cdots,i_{r_0}} Z_n} } \qquad \cref{eq:directional:algebraic-iteration}
                                           \\&\le \tfrac{2^{r_0}}{\prod_{l=0}^{r_0-1}(k-2l)} \sum_{i_1,\cdots,i_{r_0}=1}^\infty \norm*{ \Gamma[F_n,\partial_{i_1,\cdots,i_{r_0}} F_n]\partial_{i_1,\cdots,i_{r_0}} Z_n}_{L^{1}(\prob)} \qquad \cref{eq:directional:algebraic-reverse}
        \end{split}
      \end{equation}
      Applying then the Cauchy--Schwarz inequality twice: first on the $L^{1}$-norm then on the sum, we obtain
      \begin{equation}\label{eq:influence:cs}
        \abs*{ \Esp*{ \Gamma(F_{n}, F_{n}) Z_{n} } } 
        \le \tfrac{2^{r_0} }{\prod_{l=0}^{r_0-1}(k-2l)} \rho_m(F_n) \norm{\mathsf{D}^{r_{0}} F_{n}}_{\ell^{2}(\mathbb{N}^{r_{0}}) \otimes L^{2}(\prob)} \norm{\mathsf{D}^{r_{0}} Z_{n}}_{\ell^{2}(\mathbb{N}^{r_{0}}) \otimes L^{2}(\prob)}
      \end{equation}
      where we use the random sequences
      \begin{equation*}
        \mathsf{D}^{r_{0}} X \coloneq \set*{ \partial_{i_{1},\dots,i_{r_{0}}} X : (i_{1}, \dots, i_{r_{0}}) \in \mathbb{N}^{r_{0}} }, \quad X \in \mathbb{D}^{\infty}.
      \end{equation*}
      When $F$ is a Wiener chaos of order $\lambda$, then:
      \begin{enumerate}
        \item by integration by parts \cref{eq:wiener:ipp}, $\sum_{i=1}^\infty\|\partial_i F\|_2^2 = \Esp*{\Gamma[F,F]} = -\lambda \Esp*{F^2}$;
        \item for any multi-index $(i_1,\cdots,i_{r})$, $\partial_{i_1,\cdots,i_{r}} F_n$ is a Wiener chaos of order $\lambda-r$.
      \end{enumerate}
      Combining these two facts two facts together gives
      \begin{align*}
        & \norm{\mathsf{D}^{r_{0}} F_{n}}^{2}_{\ell^{2}(\mathbb{N}^{r_{0}}) \otimes L^{2}(\prob)} = \sum_{i_1,\cdots,i_{r_0}=1}^\infty \Esp*{\paren*{\partial_{i_1,\cdots,i_{r_0}} F_n}^{2}} = p(p-1)\cdots(p-r+1) \Esp*{F_n^2}=\frac{p!}{(p-r_0)!},
        \\& \norm{\mathsf{D}^{r_{0}} Z_{n}}^{2}_{\ell^{2}(\mathbb{N}^{r_{0}}) \otimes L^{2}(\prob)} = \sum_{i_1,\cdots,i_{r_0}=1}^\infty \Esp*{\paren*{\partial_{i_1,\cdots,i_{r_0}} Z_n}^{2}} = \frac{p!}{(p-r_0)!} \Esp*{Z_n^2}.
      \end{align*}
      Substituting in \cref{eq:influence:cs}, we get, since $(Z_{n})$ is bounded,
    \begin{equation*}
    \abs*{ \Esp*{\Gamma[F_n,F_n]Z_n} } \le \frac{2^{r_0} \rho_m(F_n)}{\prod_{l=0}^{r_0-1}(k-2l)}\frac{p!}{(p-r_0)!}\|Z_n\|_2\to 0.
  \end{equation*}
\item \textbf{Case $k-r_0\le \frac{p}{2}$}.
      Since $k\ge 1$, $p+(k-r_0)>p-r_0$ and one may write relying on \cref{th:algebraic-lemma} that
      \begin{equation*}
        \Esp*{F_n \partial_{i_1,\cdots,i_{r_0}} F_n\partial_{i_1,\cdots,i_{r_0}} Z_n} = \frac{2}{k}\Esp*{\Gamma[F_n,\partial_{i_1,\cdots,i_{r_0}} Z_n]\partial_{i_1,\cdots,i_{r_0}} F_n}.
      \end{equation*}
      The rest of the proof is identical up to a change in the constant $\frac{2}{2p-k}$ that is now replaced by $\frac{2}{k}$.
      The final bound is then given by
    \begin{equation*}
    \abs*{ \Esp*{\Gamma[F_n,F_n]Z_n} } \le \frac{2^{r_0}(2p-k) \rho_m(F_n)}{k\prod_{l=0}^{r_0-1}(k-2l)}\frac{p!}{(p-r_0)!}\|Z_n\|_2\to 0.
  \end{equation*}
\end{enumerate}
\end{proof}

\begin{remark}\label{rem-clt-3d}
  When $p=3$, then $\floor*{ \frac{p}{2} } =1$ and, owing to the fact that every element of $\mathcal{W}_{1}$ is of the form $\vec{a} \cdot \vec{G}$ for some $\vec{a} \in \ell^{2}(\mathbb{N})$, we get the following equivalence for Wiener chaoses of degree $3$
  \begin{equation*}
    F_n \xrightarrow[n \to \infty]{\law} \mathcal{N}(0,1) \Leftrightarrow \sup_{\norm{\vec{a}}_{\ell^{2}(\mathbb{N})} \leq 1 } \norm*{ \sum a_{i} \partial_{i} F_{n} }_{L^{2}(\prob)} 
  \end{equation*}
  For chaoses of higher degree, $\rho_1$ is not enough to measure asymptotic normality as illustrated by the following counterexample.
    Set $F_n \coloneq (n^{-1/2} \sum_{k=1}^n H_2(G_{2k})) \times (n^{-1/2} \sum_{k=1}^n H_2(G_{2k+1}))$, which, by the usual central limit theorem converges to the product of two independent Gaussian.
    For $X \coloneq \vec{a} \cdot \vec{G}$ with $\norm{\vec{a}}_{\ell^{2}(\mathbb{N})} = 1$, we find
    \begin{equation*}
      \Gamma[F_n,X] = \paren*{\frac{2}{\sqrt{n}}\sum_{k=1}^n a_{2k} G_{2k}} \times \paren*{\frac{1}{\sqrt{n}} \sum_{k=1}^n H_2(G_{2k+1})}
      + \paren*{\frac{1}{\sqrt{n}}\sum_{k=1}^n H_2(G_{2k})} \times \paren*{\frac{2}{\sqrt{n}}\sum_{k=1}^n a_{2k+1} G_{2k+1}}.
    \end{equation*}
    Moreover, 
    \begin{equation*}
      \norm*{ \frac{1}{\sqrt{n}}\sum_{k=1}^n a_{2k} G_{2k} }_{L^{2}(\prob)}^2= \frac{1}{n} \sum_{k=1}^n a_{2k}^{2} \le \frac{1}{n},
    \end{equation*}
    with a similar bound for the odd terms.
    This implies, computing the norm with the previous equality and using independence, that $\|\Gamma[F_{n},X]\|_2\le \frac{2\sqrt{2}}{\sqrt{n}}$.
    Thus $\rho_1(F_n)\to 0$ but the limit is non Gaussian.
\end{remark}

\subsubsection{Asymptotic independence and normal convergence}
We show how our new criterion provides new insights on asymptotic independence for Wiener chaoses that are asymptotically Gaussian.
We start by establishing that the carré du champ of an asymptotically normal Wiener chaos is asymptotically an eigenfunction.

\begin{lemma}\label{th:wiener:gamma-asymptotic-chaos}
  Let $(F_n,X_n)_{n\ge 1}$ be a sequence in $\mathcal{W}_p\times \mathcal{W}_q$ with , $\Esp*{F_n^2}=1$ and $X_n\to \mathcal{N}(0,1)$.
  Then,
  \begin{equation}\label{eq:proj-gamma}
    \begin{cases}
    & \text{if}\ q\le p, \qquad (\mathsf{L}+p-q)\Gamma[F_n,X_n] \xrightarrow[n \to \infty]{L^{2}(\prob)} 0.
  \\& \text{if}\ q > p, \qquad \Gamma(F_{n}, X_{n}) \xrightarrow[n \to \infty]{L^{2}(\prob)} 0.
  \end{cases}
\end{equation}
\end{lemma}

\begin{proof}
  The case $p < q$ follows immediately from \cref{th:wiener:normal-convergence}.
  Thus, take $p \geq q$.
  By \cref{eq:wiener:multiplication-gamma}, $\Gamma[F_n,X_n]\in \mathcal{W}_{\leq p+q-2}$.
  Since $\mathsf{L} + (p-q) \mathsf{I})$ corresponds to the orthogonal projection on $\mathcal{W}_{p-q}^{\perp}$, it is sufficient to consider, by \cref{eq:wiener:multiplication-gamma}, $k \in \{ 0, \dots, p+q-2 \}$ such that $k\neq p-q$ and $(Z_n)_{n\ge 1}$ a bounded sequence in $\mathcal{W}_k$, and to establish that $\Esp*{\Gamma[F_n,X_n]Z_n}\to 0$.
  If $q<p$ then $q-p<0$ and $k>q-p$, while if $p=q$ then by assumption $k\neq p-q=0$ and $k>q-p$ as well.
  Hence, applying \cref{th:algebraic-lemma} twice gives
  \begin{equation*}
    \begin{split}
      \Esp*{\Gamma[F_n,X_n]Z_n} &=\frac{p+q-k}{2}\Esp*{F_n X_n Z_n}
                              \\&=\frac{p+q-k}{2}\times\frac{2}{p+k-q}\Esp*{\Gamma[F_n,Z_n]X_n}
                              \\&=\frac{p+q-k}{p+k-q}\sum_{i=1}^\infty [\partial_i F_n \partial_i Z_n X_n].
    \end{split}
  \end{equation*}
  Iterating the previous procedure $r_0$ times with $r_0$ being the smallest integer such that we do not divide by $0$ yields
  \begin{equation*}
    \Esp*{\Gamma[F_n,X_n]Z_n}
=\frac{p+q-k}{2} \times \prod_{l=0}^{r_0-1}\frac{2}{p+k-q-2l}\sum_{i_1,\cdots,i_{r_0}=1}^\infty \Esp*{\partial_{i_1,\cdots,i_{r_0}} F_n \partial_{i_1,\cdots,i_{r_0}} Z_n X_n}.
  \end{equation*}
  Since $r_0$ is maximal, the previous expression cannot be further decomposed and one must have $p-r_0 + k-r_0 = q$.
  Hence, either $p-r_0\le \frac{q}{2}$ either $k-r_0\le\frac{q}{2}$ and, as in the previous proof, we consider both cases.
  \begin{enumerate}[(a),wide,nosep]
    \item \textbf{Case $k-r_0\le \frac{q}{2}$}.
      For a given multi-index $(i_1,\cdots,i_{r_0})$ we have $\partial_{i_1,\cdots,i_{r_0}} F_n\in\mathcal{W}_{p-r_0}$ and $\partial_{i_1,\cdots,i_{r_0}} Z_n\in\mathcal{W}_{k-r_0}$. Besides $(k-r_0)+q\neq p-r_0$ since $k\neq p-q$ hence
      \begin{equation*}
        \begin{split}
          \abs*{ \Esp*{\partial_{i_1,\cdots,i_{r_0}} F_n\partial_{i_1,\cdots,i_{r_0}} Z_n X_n} } &= \frac{2}{q+k-p} \abs*{ \Esp*{\partial_{i_1,\cdots,i_{r_0}} F_n\Gamma[X_n,\partial_{i_1,\cdots,i_{r_0}} Z_n]} }
                                                                                                 \\&\le \frac{2}{q+k-p} \norm*{ \partial_{i_1,\cdots,i_{r_0}} F_n }_{L^{2}(\prob)} \norm*{ \partial_{i_1,\cdots,i_{r_0}} Z_n }_{L^{2}(\prob)} \, \rho_{\floor*{\frac{q}{2}}}(X_n).
        \end{split}
      \end{equation*}
      Using Cauchy-Schwarz in the same way as to derive \cref{eq:influence:cs}, we find
      \begin{equation*}
        \begin{split}
          \norm{\Esp*{\Gamma[F_n,X_n]Z_n}}
&\le \tfrac{p+q-k}{2} \prod_{l=0}^{r_0-1}\tfrac{2}{p+k-q-2l} \tfrac{2}{q+k-p} \times \rho_{\floor*{ \frac q 2 } }(X_n) \psh*{ \norm*{\mathsf{D}^{r_{0}}F_{n}}_{L^{2}(\prob)}}{ \norm*{\mathsf{D}^{r_{0}} Z_{n}}_{L^{2}(\prob)} }_{\ell^{2}(\mathbb{N}^{r_{0}})}
\\&\le \rho_{\floor*{ \frac q 2}}(X_n)\tfrac{p+q-k}{q+k-p} \prod_{l=0}^{r_0-1}\tfrac{2}{p+k-q-2l}\times \norm{\mathsf{D}^{r_{0}}F_{n}}_{\ell^{2}(\mathbb{N}^{r_{0}}) \otimes L^{2}(\prob)} \norm{\mathsf{D}^{r_{0}}Z_{n}}_{\ell^{2}(\mathbb{N}^{r_{0}}) \otimes L^{2}(\prob)}
\\&=\rho_{\lfloor \frac q 2 \rfloor}(X_n)\tfrac{p+q-k}{q+k-p} \prod_{l=0}^{r_0-1}\tfrac{2}{p+k-q-2l} \sqrt{\tfrac{p!}{(p-r_0)!}} \sqrt{\tfrac{k!}{(k-r_0)!}} \times \|F_n\|_{L^{2}(\prob)} \|Z_n\|_{L^{2}(\prob)}.
        \end{split}
      \end{equation*}
      By assumption $X_n\to\mathcal{N}(0,1)$, thus using \cref{th:wiener:normal-convergence}, we get $\rho_{\lfloor\frac{q}{2}\rfloor}(X_n)\to 0$.
      Thus, the last line vanishes as $n \to \infty$, since both $(F_{n})$ and $(Z_{n})$ are bounded in $L^{2}(\prob)$.

    \item \textbf{Case $p-r_0\le \frac{q}{2}$}.
      For a given multi-index $(i_1,\cdots,i_{r_0})$ we have $\partial_{i_1,\cdots,i_{r_0}} F_n\in\mathcal{W}_{p-r_0}$ and $\partial_{i_1,\cdots,i_{r_0}} Z_n\in\mathcal{W}_{k-r_0}$. Besides $(p-r_0)+q\neq k-r_0$ since $k\le p+q-2<p+q$ hence
      \begin{equation*}
        \begin{split}
          \left|\Esp*{\partial_{i_1,\cdots,i_{r_0}} Z_n\partial_{i_1,\cdots,i_{r_0}} F_n X_n}\right|&=\frac{2}{p+q-k} \left|\Esp*{\partial_{i_1,\cdots,i_{r_0}} Z_n\Gamma[X_n,\partial_{i_1,\cdots,i_{r_0}} F_n]}\right|
                                                                                                    \\&\le \frac{2}{p+q-k}\|\partial_{i_1,\cdots,i_{r_0}} Z_n\|_2 \|\partial_{i_1,\cdots,i_{r_0}} F_n\|_2~\rho_{\lfloor\frac{q}{2}\rfloor}(X_n).
        \end{split}
      \end{equation*}
      The rest of the proof is identical to the previous case with just a minor change on the final constants.
  \end{enumerate}
\end{proof}

\begin{definition}
  For $(X_{n})$ a sequence in $\mathcal{W}_{q}$ that is asymptotically normal, we define the \emph{asymptotic independence algebra} as the set $\mathcal{A}(X_{n})$ of all sequences $(Y_{n})$ in $\mathbb{D}^{\infty}$ such that
  \begin{equation*}
    \Gamma(X_{n}, Y_{n}) \xrightarrow[n \to \infty]{L^{2}(\prob)} 0.
  \end{equation*}
\end{definition}

\begin{lemma}\label{th:independence:algebra}
  For $(X_{n})$ as above, the set $\mathcal{A}(X_{n})$ is an algebra stable by orthogonal projections on chaoses, $\mathsf{L}$, $\Gamma$, and composition with smooth functions with polynomial growth.
\end{lemma}

\begin{proof}
  The fact that $\mathcal{A}(X_{n})$ is an algebra stable by composition with smooth functions follows immediately from the fact that $\Gamma$ is bilinear and satisfies a chain rule \cref{eq:wiener:chain-rule}.
  Let us show that it is stable by projections.
  That would imply stability by $\mathsf{L}$ since $\mathsf{L}$ is a multiplication operator on chaoses, and thus by $\Gamma$ in view of the definition of $\Gamma$.
  Take $(Y_{n}) \in \mathcal{A}(X_{n})$.
  In view of \cref{th:wiener:gamma-asymptotic-chaos}:
  \begin{enumerate}[(i)]
    \item $\Gamma(X_{n}, \mathsf{J}_{r}Y_{n}) \to 0$ for $r < q$.
    \item  $\Gamma(X_{n}, \mathsf{J}_{r}Y_{n})$ are asymptotically in chaoses of different degree, for $r \geq q$.
  \end{enumerate}
  Thus writing
  \begin{equation*}
    \Gamma(X_{n}, Y_{n}) = \sum_{r \in \mathbb{N}} \Gamma(X_{n}, \mathsf{J}_{r} Y_{n}) \to 0,
  \end{equation*}
  and using the orthogonality of Wiener chaoses we conclude.
\end{proof}

We give a novel characterization of asymptotic independence in terms of the carré du champ, that justifies the name of $\mathcal{A}$.
\begin{proposition}\label{th:wiener:asymptotic-independence}
  Let $(X_{n})$ be a sequence in $\mathcal{W}_{q}$ asymptotically normal.
  Then every $(Y_{n}) \in \mathcal{A}(X_{n})$ is asymptotically independent of $(X_{n})$, namely
  \begin{equation*}
     \Esp*{ \varphi(X_{n}) \psi(Y_{n}) } - \Esp*{ \varphi(X_{n}) } \Esp*{ \psi(Y_{n}) }  \xrightarrow[n \to \infty]{} 0, \qquad \varphi,\psi \in \mathscr{C}_{b}(\mathbb{R}).
  \end{equation*}
\end{proposition}

\begin{proof}
  Take $\psi \in \mathscr{C}^{\infty}_{b}(\mathbb{R})$ and define
  \begin{equation*}
    c_{n}(t) \coloneq \Esp*{ \psi(Y_{n}) \mathrm{e}^{\mathrm{i} t X_{n}} } - \Esp*{ \psi(Y_{n}) } \mathrm{e}^{-t^{2}/2}, \qquad t \in \mathbb{R}.
  \end{equation*}
  It suffices to show that $c_{n}(t) \to 0$ for all $t \in \mathbb{R}$.
  The function $c_{n}$ is smooth and verifies $c_{n}(0) = 0$.
  Moreover, we find
  \begin{equation*}
    \dot{c}_{n}(t) = \mathrm{i} t \Esp*{ X_{n} \psi(Y_{n}) \mathrm{e}^{\mathrm{i} tX_{n}} } + t \mathrm{e}^{-t^{2}/2} \Esp*{\psi(Y_{n})}.
  \end{equation*}
  By integration by parts \cref{eq:wiener:ipp} and the chain rule \cref{eq:wiener:chain-rule}, we thus find
  \begin{equation*}
    \dot{c}_{n}(t) = \frac{\mathrm{i} t}{q} \Esp*{ \Gamma(X_{n}, \psi(Y_{n})) \mathrm{e}^{\mathrm{i} t X_{n}} } - \frac{t^{2}}{q} \Esp*{ \Gamma(X_{n}, X_{n}) \mathrm{e}^{\mathrm{i} tX_{n}} \psi(Y_{n}) } + t \mathrm{e}^{-t^{2}/2} \Esp*{ \psi(Y_{n}) }.
  \end{equation*}
  Since $(Y_{n}) \in \mathcal{A}(X_{n})$, we have that the first term vanishes.
  Moreover, by \cref{th:wiener:normal-convergence} \cref{i:wiener:normal-convergence:variance}, we have that $\Gamma(X_{n}, X_{n}) \to q$.
  Thus
  \begin{equation*}
    \dot{c}_{n}(t) = - t c_{n}(t) + o(1).
  \end{equation*}
\end{proof}

\subsection{Decomposition of variables in the direction of strongest directional influence}
In this section we gather several technical lemmas allowing us to factor directions out of a non-asymptotically normal chaos.
Thus in all this section we fix $p \geq 2$ and we let $(F_n)_{n\ge 1}$ be a sequence in $\mathcal{W}_p$ with $\Esp*{F_n^2}=1$ that is not asymptotically Gaussian.
As eluded in the introduction, we want to factorize out of $F_{n}$ \emph{macroscopic directions}.
We formalize this intuition by measuring the importance of a direction using the directional influence.
This naturally leads to the following definition.

\begin{definition}
We define the \emph{degree of strongest directional influence}
\begin{equation*}
  q(F_{n}) = q \coloneq \min\set*{ k \in \mathbb{N}^{*} : \rho_k(F_n)\ \text{does not converge to}\ 0 }.
\end{equation*}
We call any sequence $(X_n)$ realising $\rho_{q}(F_{n})$ a \emph{direction of strongest influence}.
\end{definition}

\begin{remark}
  Since $F_n\not\to\mathcal{N}(0,1)$ then $\rho_{\floor{\frac{p}{2}}}(F_n) \not\to 0$ by \cref{th:wiener:normal-convergence} \cref{i:wiener:normal-convergence:influence-min} and $q \leq \floor{\frac{p}{2}}$ is well defined as the minimum of a non empty set of positive integers.
  Up to extracting a subsequence, we always assume that $\rho_q(F_n)$ is lower bounded.
\end{remark}

The following result ensures that directions of strongest influence are asymptotically Gaussian.
\begin{lemma}\label{th:wiener:decomposition:strongest-influence-gaussian}
  For any direction of strongest influence $(X_{n})$, we have
\begin{equation*}
  X_{n} \xrightarrow[n \to \infty]{\law} \mathcal{N}(0,1).
\end{equation*}
\end{lemma}
\begin{proof}
  We prove the claim by induction on $p$.
  \begin{enumerate}[(a),wide,nosep]
    \item \textbf{Initialisation}.
  Let $p\coloneq2$ then, $\floor{\frac{p}{2}} = 1$, and thus, $q = 1$.
  Hence, every $(X_{n})$ realising $\rho_{1}$ is an element of $\mathcal{W}_{1}$ with unit variance, that is $\Law*{X_{n}} = \mathcal{N}(0,1)$ for all $n \in \mathbb{N}$.

\item \textbf{Induction step}.
Assume that $p>2$, and that we have established the claim for all $p' < p$.
Take $(X_{n})$ realising $\rho_{q}(F_{n})$, and assume, by contradiction, that it is not asymptotically Gaussian.
Since $q \leq \floor{\frac{p}{2}} < p$, we can apply the induction hypothesis, on $(X_{n})$.
Thus, up to extracting a subsequence, there exists $r \leq \floor{\frac{q}{2}}$ and a sequence $(Y_n)$ living in $\mathcal{W}_r$ such that
\begin{align}
  &\label{eq:wiener:decomposition:bound-rho} \norm*{ \Gamma[X_n,Y_n]}_{L^{2}(\prob)} \ge \rho_r(X_n)-\frac{1}{n} \geq \delta>0,
\\&\label{eq:wiener:decomposition:Y-Gaussian} Y_n \xrightarrow[n \to \infty]{\law} \mathcal{N}(0,1).
\end{align}
Consider $A_{n}$ the orthogonal projection of $X_n$ on the closed vector space $Y_n \mathcal{W}_{q-r}$.
Thus, we write $X_n=A_n Y_n+R_n$ with $A_n\in\mathcal{W}_{q-r}$ and $R_n\perp Y_n \mathcal{W}_{q-r}$.
Using \cref{th:wiener:gamma-asymptotic-chaos}, $\norm*{(\mathsf{L}+q-r)[\Gamma[X_n,Y_n]]}_{L^{2}(\prob)} \to 0$.
In particular, $Y_{n} \Gamma(X_{n}, Y_{n})$ is asymptotically in $Y_{n} \mathcal{W}_{q-r}$.
Thus, using properties of orthogonal projections, we find:
\begin{equation*}
  \Esp*{X_n Y_n \Gamma[X_n,Y_n]} - \Esp*{A_n Y_n \Gamma[X_n,Y_n]} \xrightarrow[n \to \infty]{} 0.
\end{equation*}

On the other hand, combining \cref{th:wiener:spectral-inequality,eq:wiener:decomposition:bound-rho},
\begin{equation*}
  \Esp*{X_n Y_n \Gamma[X_n,Y_n]} \ge \frac{p+r}{4} \Esp*{\Gamma[X_n,Y_n]^2} \ge \frac{p+r}{4} \delta^2.
\end{equation*}
Hence, the two previous equations with the Cauchy--Schwarz inequality gives
\begin{equation*}
  \frac{p+r}{4} \delta^2\le \|A_n Y_n\|_{L^{2}(\prob)} \|\Gamma[X_n,Y_n]\|_{L^{2}(\prob)},
\end{equation*}
which implies that for some $\delta'>0$ we have $\|A_n Y_n\|\ge\delta'$.
Then, since $\Esp*{F_{n}^{2}}=1$, by orthogonality of $R_n$ and $A_n Y_n$ we have also $\Esp*{R_n^{2}} \le 1-\delta'^2$.
Moreover, the chain rule \cref{eq:wiener:chain-rule} implies that 
\begin{equation*}
  \Gamma[F_n,X_n]=\Gamma[F_n,A_n] Y_n + A_n \Gamma[F_n, Y_n] + \sum_{s = 1}^{q-1} \Gamma[F_n,\mathsf{J}_{r}R_n] + \Gamma[F_n,\mathsf{J}_{q}R_n].
\end{equation*}
Assume for the moment, that $(A_n)_{n\ge 1}$ is bounded in $L^2$, then all the terms in the right-hand side vanish, but the last one, by minimality of $q$, as they all involve quantities of the form $\Gamma[F_n,Z_n]$ with $Z_n$ a bounded sequence in $\mathcal{W}_{t}$ with $t<q$.
Since $\Esp*{ R_{n}^{2}} \le 1-\delta'$, taking into account the normalization in \cref{df:directional-influence}, the last term has a norm bounded from above by $( 1-\delta'^2 )^{1/2} \rho_q(F_n)$.
Hence, letting $n\to\infty$ in the above equation yields that $\rho_q(F_n)\le ( 1-\delta'^2 )^{1/2}\rho_q(F_n)$ which is absurd, hence the result.

It remains to prove that $(A_n)_{n\ge 1}$ is a $L^{2}(\prob)$-bounded sequence.
To do so, we combine the equivalence of norms \cref{eq:wiener:norm-equivalence} with the fact that $Y_n\to\mathcal{N}(0,1)$, which is \cref{eq:wiener:decomposition:Y-Gaussian}.
We thus have
\begin{eqnarray*}
  \|A_n\|_{L^{2}(\prob)} &\le \frac{1}{M}\|A_n Y_n\|_{L^{2}(\prob)} +\|A_n 1_{|Y_n|<M}\|_{L^{2}(\prob)}
                         \\&\le \frac{1}{M}+3^{\frac r 2}\|A_n\|_{L^{2}(\prob)} \Prob*{ |Y_n|\le M }^{\frac 1 4}.
\end{eqnarray*}
Since, $\Prob*{ |Y_n|\le M }\to \Prob*{ |N|\le M }$ for $N\sim\mathcal{N}(0,1)$, for $M$ small enough fixed and $n$ large enough one has $3^{\frac{r}{2}} \Prob*{|Y_n|\le M}^{\frac 1 4}\le \frac{1}{2}$ which by the above inequality gives that $\|A_n\|_{L^{2}(\prob)} \le \frac{2}{M}$ for $n$ large enough, hence the result.
\end{enumerate}
\end{proof}

We now present the decomposition procedure, which serves as the cornerstone of our proof by induction.
\begin{lemma}\label{th:wiener:decomposition:once}
  Let $p \in \mathbb{N}^{*}$ and $(F_n)_{n\ge 1}$ be a sequence in $\mathcal{W}_p$ with $\Esp*{F_n^2}=1$.
  Take an asymptotically normal sequence $(X_n)_{n\ge 1}$ in $\mathcal{W}_q$ with $q \in \{1, \dots, p-1\}$.
Then, the following decomposition holds
\begin{equation}\label{eq:wiener:decomposition:once}
  F_n = \sum_{l=1}^{\floor{\frac{p}{q}}} A_{n,l} H_{l}(X_n)+A_{n,0},
\end{equation}
where, for every $l \in \set*{ 0, \dots,\floor*{\frac{p}{q}} }$, $A_{n,l} \in \mathcal{W}_{\leq (p - lq)}$ is
\begin{enumerate}[(i)]
  \item {\bf asymptotically independent:} $(A_{n,l}) \in \mathcal{A}(X_{n})$, namely 
  \begin{equation*}
    \Gamma[A_{n,l},X_n]\ \xrightarrow[n \to \infty]{L^{2}(\prob)} 0.
  \end{equation*}
\item {\bf asymptotically chaoses:}
  \begin{equation*}
    (\mathsf{L}+p-lq)A_{n,l} \xrightarrow[n \to \infty]{L^{2}(\prob)} 0.
\end{equation*}
\end{enumerate}
\end{lemma}

\begin{proof}
  Introduce the vector space $E_n$ to be the $L^{2}(\prob)$-closure of $X_n \mathcal{W}_{p-q} +{\mathcal{W}_{\leq (p-1)}}$.
 Let $P_{n}$ be the orthogonal projection of $F_n$ on $E_n$.
 Write $F_n= P_{n} + R_n$ where $R_n \in E_{n}^{\perp}$.
 We proceed by induction with some preparatory steps.

 \begin{enumerate}[(a),wide,nosep]
   \item 
     \textbf{Asymptotic independence of $(R_{n})$}.
     Using \cref{eq:wiener:multiplication}), $X_n \mathcal{W}_{p-q}\subset \mathcal{W}_{\leq p}$ which implies that $E_n\subset \mathcal{W}_{\leq p}$, and thus $R_n=F_n-P_{n} \in \mathcal{W}_{\leq p}$.
     Besides, since $\mathcal{W}_{\leq (p-1)} \subset E_{n}$, we have $(E_n)^{\perp}\subset \mathcal{W}_{\leq (p-1)}^{\perp}$ guaranteeing that $R_n \in \mathcal{W}_{\leq p} \cap \mathcal{W}_{\leq (p-1)}^{\perp} = \mathcal{W}_p$.
     Assume by contradiction that $\Gamma[R_n,X_n]\not\to 0$.
     Up to extracting a subsequence one finds $\delta>0$ such that $\|\Gamma[R_n,X_n]\|\ge \delta$.
     By definition of orthogonal projection, and \cref{th:algebraic-lemma}, we get
\begin{equation*}
\Esp*{R_n X_n A} = 0
= \frac{1}{q}\Esp*{\Gamma[R_n,X_n]A}.
\end{equation*}
On the other hand, using \cref{th:wiener:gamma-asymptotic-chaos}, we infer that $\Gamma[R_n,X_n]$ asymptotically belongs to $\mathcal{W}_{p-q}$.
Choosing $A \coloneq \mathsf{J}_{p-q}(\Gamma[R_n,X_n])$, which by \cref{th:wiener:gamma-asymptotic-chaos} is $\Gamma[R_n,X_n] + o_{L^{2}(\prob)}(1)$, we find that for $n$ large enough that $\Esp*{R_n X_n A}\ge \frac{\delta^2}{2}>0$ which is contradictory.
Thus, we deduce that $\Gamma[R_n,X_n]\to 0$

\item \textbf{Asymptotic chaos from asymptotic independence}.
  Let us assume that we are given a decomposition \cref{eq:wiener:decomposition:once} with $(A_{n,l}) \in \mathcal{A}(X_{n})$ and let us show that they are automatically asymptotically chaotic.
  Since $F_{n} \in \mathcal{W}^{p}$ applying $(\mathsf{L} + p\mathsf{I})$ to \cref{eq:wiener:decomposition:once}, and developing $\mathsf{L}(AX)$ using the definition of $\Gamma$ we find that
  \begin{equation*}
  0 = \sum_{l=1}^{\floor{\frac{p}{q}}} \bracket[\Big]{2 \Gamma(A_{n,l}, H_{l}(X_{n})) + A_{n,l} (\mathsf{L} + lq) H_{l}(X_{n}) + H_{l}(X_{n}) (\mathsf{L} + p - lq) A_{n,l} } + (\mathsf{L} + p) A_{n,0}.
  \end{equation*}
  Since $A_{n,l} \in \mathcal{A}(X_{n})$, by the chain rule \cref{eq:wiener:chain-rule}, $\Gamma(A_{n,l}, H_{l}(X_{n})) \to 0$, in $L^{2}(\prob)$; while, by \cref{th:wiener:normal-convergence} \cref{i:wiener:normal-convergence:eigenvalue} $(\mathsf{L} + lq) H_{l}(X_{n}) \to 0$ also in $L^{2}(\prob)$.
  On the other hand, $(A_{n,l}) \in \mathcal{A}(X_{n})$ thus by \cref{th:wiener:asymptotic-independence} the $A_{n,l}$'s are asymptotically independent of $(X_{n})$.
  In particular, using that asymptotically the $H_{l}(X_{n})$ are orthogonal for different values of $l$, we find, taking the variance in the above equation that
  \begin{equation*}
    0 = \sum_{l=0}^{\floor{\frac{p}{q}}} \Esp*{ \paren*{ (\mathsf{L} + lq) A_{n,l}}^{2} } + o(1).
  \end{equation*}
  This shows that the $A_{n,l}$'s are indeed asymptotically chaotic.
\end{enumerate}
It remains to show that a decomposition such as \cref{eq:wiener:decomposition:once} exists a that it satisfies the asymptotic independence, which we prove by induction.

\begin{enumerate}[(a),wide,nosep,resume]
\item \textbf{Initialisation}.
  Let $p \coloneq 2$. 
Since $q<p$, we have $q=1$.
Considering as before the orthogonal projection of $F_n$ on $E_{n}$, the closure of  $X_n\mathcal{W}_1 + \mathcal{W}_{\leq 1}$, we write $F_{n} = P_{n} + R_{n}$.
The previous step ensures that $\Gamma(F_{n}, R_{n}) \to 0$.
Since $X_{n} \in \mathcal{W}_{1}$ we actually have that
\begin{equation*}
X_{n} \mathcal{W}_{1} + \mathcal{W}_{\leq 1} = H_{2}(X_{n}) \mathbb{R} \oplus X_{n} (X_{n}^{\perp} \cap \mathcal{W}_{\leq 1}) \oplus (X_{n}^{\perp} \cap \mathcal{W}_{\leq 1}),
\end{equation*}
which is already closed.
Thus $P_{n} = a_{n} H_{2}(X_{n}) + X_{n} N_{n} + B_{n}$ for some $a_{n} \in \mathbb{R}$, $N_{n} \in \mathcal{W}_{\leq 1}$ independent of $X_{n}$ and $B_{n} \in \mathcal{W}_{\leq 1}$, also independent of $X_{n}$.
We claim that setting $A_{n,2} \coloneq a_{n}$, $A_{n,1} \coloneq N_{n}$, and $A_{n,0} \coloneq B_{n} + R_{n}$ yields the desired decomposition.
Indeed, since $N_{n}$, $B_{n}$ and $a_{n}$ are independent of $X_{n}$, we get that
\begin{equation*}
  \Gamma(X_{n}, a_{n}) = \Gamma(X_{n}, N_{n}) = \Gamma(X_{n}, B_{n}) = 0,
\end{equation*}
Since $\Gamma(X_{n}, R_{n}) \to 0$, we have that the asymptotic independence is satisfied, by \cref{th:wiener:asymptotic-independence}.

\item \textbf{Induction step}.
Take $p > 2$ and assume that the decomposition is proved for all $p' < p$.
In this case, although it must be possible to decompose $E_{n}$ in a direct sum as in the initialization step, since we are dealing with Wiener chaoses of degree more than $1$, orthogonality and independence are not the same and that would render finding this orthogonal decomposition more tedious, thus we adopt a slightly different strategy.
Since we are projecting on $E_{n}$ that is a closure, we write $F_n = X_n A_n + B_n + R_n+o(1)$ where $A_n \in \mathcal{W}_{p-q}$, $B_{n} \in \mathcal{W}_{\leq (p-1)}$, and where the $o(1)$ represent the approximation of the closure by elements of $E_{n}$.
Since $p-q<p$, we apply our induction hypothesis on $A_{n}$, this yields coefficients $C_{n,l} \in \mathcal{A}(X_{n})$ such that 
\begin{equation*}
  A_n=\sum_{l=1}^{\floor*{\frac{p-q}{q}}}C_{n,l} H_l(X_n).
\end{equation*}
Then using well-known recursive formulas for Hermite polynomials, we get
\begin{equation*}
  A_n X_n = \sum_{l=2}^{\floor{\frac{p-q}{q}}+1}C_{l-1,n}H_l(X_n) + \sum_{l=0}^{\floor{ \frac{p-q}{q}} -1}(l+1) C_{l+1,n}H_l(X_n).
\end{equation*}
We handle $B_{n}$ similarly by considering its different chaotic decomposition.
As a result, one can decompose $A_n X_n +B_n$ as linear combinations of $(H_l(X_n))_l$ with coefficients in $\mathcal{A}(X_{n})$.
The maximal index $l$ involved in this decomposition is given by $\floor{\frac{p-q}{q}}+1\le \floor{\frac{p}{q}}$.
\end{enumerate}
\end{proof}

Since by \cref{th:wiener:decomposition:strongest-influence-gaussian} any direction of strongest influence is asymptotically normal, these directions are natural candidates to apply \cref{eq:wiener:decomposition:once}.
We now establish that iterating this decomposition is a compatible with our influence-based criterion.

\begin{lemma}\label{th:wiener:decomposition:influence}
  Apply the decomposition \cref{eq:wiener:decomposition:once} with $(X_{n})$ a direction of strongest influence.
   Then, for some constant $c_{p,q}>0$ only depending on $(q,p)$, we have $\|A_{0,n}\|_{L^{2}(\prob)} \le 1-c_{p,q}\delta$.
  Moreover,
  \begin{equation*}
    \rho_{r}(A_{0,n}) \xrightarrow[n \to \infty]{} 0, \qquad r < q.
  \end{equation*}
\end{lemma}

\begin{remark}
  Since $r < q$, we have by minimality of $q$ that $\rho_{r}(F_{n}) \to 0$.
Thus, this result states that if $\rho_{r}(F_{n}) \to 0$, then $\rho_{r}(A_{n,0}) \to 0$.
\end{remark}

\begin{proof}
  Take $(Z_{n})$ a sequence in $\mathcal{W}_{r}$.
  Taking the carré du champ in \cref{eq:wiener:decomposition:once} yields
  \begin{equation*}
    \Gamma(F_{n}, Z_{n}) = \sum_{l=1}^{\floor{\frac{p}{q}}} A_{n,l} H_{l}'(X_{n}) \Gamma(X_{n}, Z_{n}) + \sum_{l=1}^{\floor{\frac{p}{q}}} H_{l}(X_{n}) \Gamma(A_{n,l}, Z_{n}) + \Gamma(A_{n,0}, Z_{n}).
  \end{equation*}
  By minimality of $q$, the left-hand side vanishes as $n \to \infty$.
  Since $(X_{n})$ is asymptotically Gaussian by \cref{th:wiener:decomposition:strongest-influence-gaussian}, we find by \cref{th:wiener:normal-convergence} \cref{i:wiener:normal-convergence:influence-max} that the first term in the right-hand side also vanishes.
  Thus
  \begin{equation*}
    o_{L^{2}(\prob)}(1) = \sum_{l=1}^{\floor{\frac{p}{q}}} H_{l}(X_{n}) \Gamma(A_{n,l}, Z_{n}) + \Gamma(A_{n,0}, Z_{n}).
  \end{equation*}
  Using that by \cref{th:wiener:normal-convergence} \cref{i:wiener:normal-convergence:eigenvalue}, $H_{l}(X_{n})$ is asymptotically a chaos, and the orthogonality of Wiener chaos, this yields
  \begin{equation*}
    \Gamma(A_{n,l}, Z_{n}) \xrightarrow[n \to\infty]{L^{2}(\prob)} 0, \qquad l \in \set*{0,\dots, \floor{\frac{p}{q}} }.
  \end{equation*}
\end{proof}

We also establish that iterating our construction preserves the asymptotic independence.
\begin{lemma}\label{th:wiener:decomposition:independence}
Let $(R_{n})$ sequence in $\mathcal{W}_{p}$ that is not asymptotically Gaussian.
Let $(X_{n,1})$ an asymptotically Gaussian sequence in $\mathcal{W}_{q_{1}}$.
Let $(X_{n,2})$ a sequence in $\mathcal{W}_{q_{2}}$ a direction of strongest influence for $(R_{n})$.
Apply the decomposition \cref{eq:wiener:decomposition:once} to $R_{n}$ that is
\begin{equation*}
  R_{n} = \sum_{l=0}^{\floor*{\frac{p}{q_{2}}}} A_{n,l} H_{l}(X_{n,2}).
\end{equation*}
If $(R_{n}) \in \mathcal{A}(X_{n,1})$.
Then, $(X_{n,2})$ and $(A_{n,l}) \in \mathcal{A}(X_{n,1})$.
\end{lemma}

\begin{proof}
  If $q_{1} \ne q_{2}$ then by \cref{th:wiener:normal-convergence}, we find that $(X_{n,2}) \in \mathcal{A}(X_{n,1})$.
  If $q_{1} = q_{2}$, using orthogonal projections, write $X_{n,2} = a_{n} X_{n,1} + Z_{n}$, where $\Esp*{Z_{n} X_{n,1}} = 0$.
  Then,
  \begin{equation*}
    \Gamma(R_{n}, X_{n,2}) = a_{n} \Gamma(R_{n}, X_{n,1}) + \Gamma(R_{n}, Z_{n}),
  \end{equation*}
  where the first term on the right-hand side vanishes by assumption.
  Since $(Y_{n})$ is a directional of strongest influence, and that $Z_{n} \in \mathcal{W}_{q_{2}}$ with $\Esp*{Z_{n}^{2}} = 1 - a_{n}^{2}$, we find
  \begin{equation*}
    \rho_{q_{2}}(F_{n}) \leq \sqrt{1-a_{n}^{2}} \rho_{q_{2}}(F_{n}) + \frac{1}{n}.
  \end{equation*}
  Thus $a_{n} \to 0$ which implies that $\Gamma(X_{n,1}, X_{n,2}) \to 0$.

  Now let us handle, the $A_{n,l}$'s.
  Here, we have no assumption on $q_{1}$ and $q_{2}$.
  Write
  \begin{equation*}
    \Gamma(R_{n}, X_{n,1}) = \sum_{l} A_{n,l} \Gamma(H_{l}(X_{n,2}), X_{n,1}) + \sum_{l} H_{l}(X_{n,2}) \Gamma(A_{n,l}, X_{n,1}).
  \end{equation*}
  By assumption, the right-hand side vanishes as $n \to 0$.
  Also, since, we have shown that $\Gamma(X_{n,1}, X_{n,2})$ vanishes with $n \to \infty$, the first sum on the right-hand side also vanishes.
  Finally, by construction $(A_{n,l}) \in \mathcal{A}(X_{n,2})$, and form the previous point $(X_{n,1}) \in \mathcal{A}(X_{n,2})$.
  It follows that, by \cref{th:independence:algebra}, $\Gamma(A_{n,l}, X_{n,1}) \in \mathcal{A}(X_{n,1})$.
  Using that the Hermite polynomials are an orthonormal basis for the Gaussian measure and that $X_{n,2}$ is asymptotically Gaussian, we find
  \begin{equation*}
    o_{L^{2}(\prob)}(1) = \sum_{l} \Esp*{ \Gamma(A_{n,l}, X_{n,2})^{2} }.
  \end{equation*}
\end{proof}

We now describe the result of iterating the decomposition \cref{eq:wiener:decomposition:once}.
We define recursively, the random variables $(A_{n,k,l})$, $(X_{n,k})$, and $(R_{n,k})$ in the following way.
Set
\begin{align*}
  & \varepsilon_{n,1,} \coloneq (A_{n,0} - \mathsf{J}_{p} A_{n,0}), \qquad \text{and for $l > 1$}\ A_{n,1,l} \coloneq A_{n,l},
\\& X_{n,1} \coloneq X_{n},
\\& F_{n,1} \coloneq \sum_{l=1}^{\floor{\frac{p}{q}}} A_{n,1,l} H_{l}(X_{n}) + A_{n,1,0},
\\& R_{n,1} \coloneq \mathsf{J}_{p} A_{n,0}.
\end{align*}
Since $(A_{n,0})$ is asymptotically in $\mathcal{W}_{p}$, we have in fact that $\varepsilon_{n,1}$ is negligible when $n \to \infty$.
By construction, $R_{n,1} \in \mathcal{W}_{p}$ with $\Esp*{R_{n,1}^{2}} \leq 1 - c\delta$.
If $(R_{n,1})$ is asymptotically Gaussian, we stop; otherwise we apply the decomposition \cref{eq:wiener:decomposition:once} to $R_{n,1}$ which produces $A_{n,2,l}$, $X_{n,2}$, $F_{n,2}$ and $R_{n,2}$.
Repeating the procedure, we obtain the following result.

\begin{lemma}\label{th:decomposition:iteration}
  There exist an interval $K \subset \mathbb{N}$, possibly infinite, integers $(q_{k} : k \in K$), and for every $k \in K$ random sequences 
  \begin{itemize}
    \item $(R_{n,k})_{n}$ in $\mathcal{W}_{p}$, 
    \item $(X_{n,k})_{n}$ in $\mathcal{W}_{q_{k}}$ for some $q_{k}$,
    \item $(A_{n,k,l})_{n}$ in $\mathcal{W}_{\leq (p-lq_{k})}$ for $l \in \set*{1, \dots, \floor*{\frac{p}{q_{k}}} }$,
    \item $(\varepsilon_{n,k,})_{n}$ in $\mathcal{W}_{\leq (p-1)}$,
  \end{itemize}
  such that
  \begin{enumerate}[(i)]
    \item\label{i:decomposition:iteration:directions} $(X_{n,k+1})_{n}$ is a direction of strongest influence for $(R_{n,k})_{n}$, and $\Gamma(X_{n,k}, X_{n,k'}) \to 0$ for $k \ne k'$,
    \item\label{i:decomposition:iteration:asymptotic-independence} $(R_{n,k'})_{n}, (A_{n,k',l})_{n} \in \mathcal{A}((X_{n,k})_{n})$, for $k' < k$,
    \item\label{i:decomposition:iteration:negligible} $\varepsilon_{n,k} = o_{L^{2}(\prob)}(1)$,
    \item\label{i:decomposition:iteration:sum} we have $F_{n} = F_{n,1} + \dots + F_{n,k} + R_{n,k}$, where
      \begin{equation*}
        F_{n,k} \coloneq \sum_{l=1}^{\floor{\frac{p}{q_{i}}}} A_{n,k,l} H_{l}(X_{n,k}) + \varepsilon_{n,k,},
      \end{equation*}
      \item\label{i:decomposition:iteration:orthogonal} $(F_{n,1},\cdots,F_{n,k}, R_{n,k})_{n}$ forms an asymptotic orthogonal sequence.
  \end{enumerate}
\end{lemma}

\begin{proof}
  As mentioned above we iterate \cref{eq:wiener:decomposition:once}.
  $K$ is simply given by the interval that contains all the $k$ until we stop the procedure, that is until $(R_{n,k})_{n}$ is asymptotically Gaussian, which might never happen in which case $K = \mathbb{N}$.
  From the construction and successive iteration of \cref{th:wiener:decomposition:independence}, we have already that \cref{i:decomposition:iteration:directions,i:decomposition:iteration:asymptotic-independence,i:decomposition:iteration:negligible,i:decomposition:iteration:sum} hold.
  Thus, only \cref{i:decomposition:iteration:orthogonal} remains to be proven.
  By \cref{th:wiener:decomposition:independence}, we have that for $k' < k$, $(F_{n,k}) \in \mathcal{A}(X_{n,k'})$.
  Thus
  \begin{equation*}
    \Esp*{ F_{n,k} F_{n,k'} } = \sum_{l=1}^{\floor*{\frac{p}{q_{k'}}}} \Esp*{A_{n,k',l} F_{n,k} H_{l}(X_{n,k'}) }.
  \end{equation*}
  The sum above vanishes since $(A_{n,k',l}F_{n,k}) \in \mathcal{A}(X_{n,k'})$.
  Finally, since $(R_{n,k}) \in \mathcal{A}(X_{n,k'})$, we also get that $\Esp*{F_{n,k'} R_{n,k}} \to 0$.
\end{proof}

We need a final ingredient for our proof by induction.
At this point, all the $F_{n,k}$ are decomposed in sum of chaoses of lower degrees that we could handle by induction.
However, $R_{n,k}$ is still in $\mathcal{W}_{p}$, to overcome this difficulty we show that it can be taken asymptotically Gaussian.
When $K$ is unbounded, for a fixed $k$, $(R_{n,k})$ might fail to be asymptotically Gaussian.
Our idea is to use that for $k_{n} \to \infty$ then we can restore the asymptotic normality.
The following result shows that up to taking an sufficiently slowly growing $(k_{n})$, we have that $(R_{n,k_{n}})$ is asymptotically Gaussian.
\begin{lemma}\label{th:wiener:decomposition:remainder-gaussian}
  Up to extraction, there exists a sequence of integers $(k_{n})$ such that $k_{n} \to \sup K$ and built upon \cref{th:analytical-lemma} such that we have that $(R_{n,k_{n}})$ is asymptotically Gaussian.
\end{lemma}

\begin{proof}
  There is only something to prove whenever $K$ is unbounded.
  Otherwise, any converging sequence $(k_{n})$ is eventually constant equals to $k^{*} \coloneq \max K$ and in this case $(R_{n,k^{*}})$ is asymptotically normal by construction.
  Thus, we assume that $K = \mathbb{N}$.
  First, by definition $\rho_{0}(F_{n}) = 0$ and $\rho_{q}(F_{n}) \geq \delta > 0$.
  By \cref{th:directional-influence-monotone}, we can take consider $s$ the largest integer such that $\rho_{s}(F_{n}) \to 0$.
  To conclude, it suffices to show that for any $(k_{n})$ converging to $\infty$, then $\rho_{s+1}(R_{n,k_{n}}) \to 0$.
  Thus, applying this fact a finite number of times, would yield that $\rho_{\floor{\frac{p}{2}}}(R_{n,k_{n}}) \to 0$, and we conclude by \cref{th:wiener:normal-convergence}.
To prove the claim, recall that for all $k \in \mathbb{N}$ we have constructed $(X_{n,k+1})$ in $\mathcal{W}_{q_{k+1}}$ that is a direction of strongest influence for $(R_{n,k})$.
Thus
\begin{equation*}
  \rho_{q_{k_{n}+1}}(R_{n,k_{n}}) \leq \frac{1}{n} + \norm*{ \Gamma(R_{n,k_{n}}, X_{n,k_{n}+1}) }_{L^{2}(\prob)} \leq \frac{1}{n} + C \norm{F_{n,k_{n}+1}}_{L^{2}(\prob)} + o_{L^{2}(\prob)}(1),
\end{equation*}
where we use that $R_{n,k_{n}} = F_{n,k_{n}+1} + R_{n,k_{n}+1}$, that $\Gamma(R_{n,k_{n}+1}, X_{n,k_{n}}) \to 0$, and that by equivalence of norms \cref{eq:wiener:norm-equivalence} and the asymptotic independence \cref{i:decomposition:iteration:asymptotic-independence} in \cref{th:decomposition:iteration} , $\Gamma(F_{n,k_{n}+1}, X_{n,k_{n}})$ has a norm comparable to that of $F_{n,k_{n}}$.
Up to extraction, we have that
\begin{equation*}
  \norm*{F_{n,k}}_{L^{2}(\prob)} \xrightarrow[n \to \infty]{} \sigma_{k}.
\end{equation*}
In view of the orthogonal property \cref{i:decomposition:iteration:orthogonal} in \cref{th:decomposition:iteration}, we find that $\sum_{k} \sigma_{k}^{2} \leq 1$.
Thus $\sigma_{k_{n}+1} \to 0$ for any $k_{n} \to \infty$, which shows that $\rho_{q_{k_{n}}}(R_{n,k_{n}}) \to 0$.
Moreover in view of \cref{th:analytical-lemma}, we have that $\norm{F_{n,k_{n}+1}}_{L^{2}(\prob)} \leq c \sigma_{k_{n}+1} + o(1)$.
Since by construction $k_{n} \geq s+1$, we conclude by \cref{th:directional-influence-monotone}.
\end{proof}

\subsection{Main result through an induction procedure}
To state properly our results, we consider convergence in law for infinite random vectors with Wiener chaotic components.
Here, we understand the convergence in the usual sense of convergence of all finite dimensional marginals, but with the additional requirement that the limit can be represented as an infinite random vector whose components belong to the Wiener space.

\begin{definition}\label{infinite-Wiener-convergence}
  Let $(\vec{F}_n) =(F_{n,1},\cdots,F_{n,i},\cdots)_{n\ge 1}$ such that for all $i\ge 1$, there exists $p_{i} \in \mathbb{N}$ such that $F_{n,i} \in \mathcal{W}_{p_i}$ for all $n \in \mathbb{N}$.
  We say that $(\vec{F}_n)_{n\ge 1}$ \emph{Wiener-converges in law} provided there exists $\vec{F}_{\infty} = (F_{\infty,1}, F_{\infty,2}, \dots)$ such that for all $n \in \mathbb{N}$ $F_{\infty,i} \in \mathcal{W}_{\leq p_{i}}$, and
\begin{equation*}
  \vec{F}_{n} \xrightarrow[n \to \infty]{\law} \vec{F}_{\infty}.
\end{equation*}
\end{definition}

\begin{remark}
  Our \cref{th:stabilite-wiener} can thus be rephrased by saying that convergence in law in equivalent to Wiener-convergence in law.
\end{remark}

Let us introduce the following definition that helps to formalize our proof.
\begin{definition}
  For $d \in \mathbb{N}$.
  We say that a sequence $(\vec{F}_{n})_{n}$ of infinite-random vector is \emph{Wiener admissible of degree $d$}, provided there exists $p \in \mathbb{N}$ and integers $p_{i} \leq p$ such that $F_{n,i} \in \mathcal{W}_{p_{i}}$, at least one of the following condition holds
  \begin{enumerate}[(i)]
    \item $(F_{n,i})$ is asymptotically Gaussian;
    \item or, $p_{i} \leq d$.
  \end{enumerate}
\end{definition}

We now complete the proof of \cref{th:stabilite-wiener} by proving the following result by induction.
\begin{theorem}
  Let $d \in \mathbb{N}$ and let $(\vec{F})_{n}$ be $d$-admissible.
  Assume that $(\vec{F}_{n})$ converges in law, then $(\vec{F}_{n})$ Wiener-converges in law.
  Moreover, notion $\vec{F}_{\infty}$ a vector realising the Wiener convergence in law, we have that $F_{\infty,i}$ is Gaussian if and only if $F_{\infty,i} \in \mathcal{W}_{1}$.
\end{theorem}

\begin{proof} 
  We work by induction on $d$.
  We have already proved the initialisation, that is $d=1$ (actually also $d=2$ and $d=3$) in \cref{th:wiener:stabilite:outline}.
Take $d \ge 2$, and assume the claim is proved for all $d' < d$.
Let $(\vec{F}_n)$ be $d$-admissible and converging in distribution, without loss of generality we further assume that $\Esp*{F_{n,i}^2}=1$ for $i$ and $n$.
Let us define $I \coloneq \{i\ge 1\,|\, F_{n,i}\not\to\mathcal{N}(0,1)\}$.
The indices not in $I$ are covered immediately by the induction hypothesis and their limits are in the first Wiener chaos.
Now for each $i\in I$ one apply our decomposition lemma \cref{th:decomposition:iteration}to $F_{n,i}$ and we obtain that there exists a $k^{*}_{i} \in \mathbb{N} \cap \{\infty\}$ such that for all $i \in \mathbb{N}^{*}$ and all $k \in \mathbb{N}$ with $k \leq k_{i}^{*}$
\begin{equation*}
  F_{n,i} = F_{n,i,1} + \dots + F_{n,i,k} + R_{n,i,k},
\end{equation*}
where $F_{n,i,k} =\sum_{l\ge 0} A_{n,i,k,l}H_{l}(X_{n,i,k})$.
Let us write $q_{i,k}$ for the integer such that $X_{n,i,k} \in \mathcal{W}_{q_{i,k}}$, and for short let $l_{i,k} \coloneq \floor{\frac{p}{q_{i,k}}}$.
Up to extraction, there exists real numbers $(\sigma_{i,k})$ such that $\norm{F_{n,i,k}}_{L^{2}(\prob)} \to \sigma_{i,k}$.
By \cref{th:analytical-lemma}, we can find $k_{n,i} \to k_{i}^{*}$, such that
\begin{equation*}
  \sum_{k=1}^{r_{n,i}} \abs*{ \Esp*{ F_{n,i,k}^{2} } - \sigma_{i,k}^{2} } \xrightarrow[n \to \infty]{} 0, \qquad i \in \mathbb{N}.
\end{equation*}
Consider the vectors
\begin{equation*}
  \vec{V}_{n} \coloneq \paren*{A_{n,i,k,l},X_{n,i,k},R_{n,i,k_{n,i}} : i \in \mathbb{N}^{*}, l \in \set*{1, \dots,l_{i,k} }}, \qquad n \in \mathbb{N}.
\end{equation*}
By \cref{th:wiener:decomposition:remainder-gaussian}, $(\vec{V}_{n})$ is admissible.
Thus, by the induction hypothesis, up to extraction, there exists
\begin{equation*}
  \vec{V}_{\infty} \coloneq \paren*{A_{\infty,i,k,l},X_{\infty,i,k},R_{\infty,i} : i \in \mathbb{N}^{*}, l \in \set*{ 1, \dots, l_{i,k} } },
\end{equation*}
with $A_{\infty,i,k,l} \in \mathcal{W}_{\leq (p-lq_{k,i})}$ and $R_{\infty,i}$ and $X_{\infty,i,k} \in \mathcal{W}_{1}$, such that
\begin{equation*}
  \vec{V}_{n} \xrightarrow[n \to \infty]{\law} \vec{V}_{\infty}.
\end{equation*}
Define further
\begin{equation*}
  F_{\infty,i,k} \coloneq \sum_{l=1}^{l_{i,k}} A_{\infty,i,k,l} H_{l}(X_{\infty,i,k}).
\end{equation*}
Moreover, since for $i \in \mathbb{N}^{*}$, $(F_{\infty,i,k})_{k]}$ are orthogonal by \cref{th:decomposition:iteration} \cref{i:decomposition:iteration:orthogonal}, we find that the $(F_{\infty,i,k})_{k}$ are also orthogonal, in particular the series
\begin{equation*}
  S_{\infty,i} \coloneq \sum_{k \in \mathbb{N}} F_{\infty,k,i},
\end{equation*}
exists in $L^{2}(\prob)$, and is an element of $\mathcal{W}_{\leq p_{i}}$ since the later is closed in $L^{2}(\prob)$.
Define $F_{\infty,i} \coloneq F_{\infty,i} + R_{\infty,i} \in \mathcal{W}_{\leq p_{i}}$.
We claim that
\begin{equation*}
  \vec{F}_{n} \xrightarrow[n \to \infty]{\law} \vec{F}_{\infty},
\end{equation*}
which would conclude the proof.
This is not completely immediate since $S_{\infty,i}$ is an infinite series.
However, we can use again the properties of $(k_{n})$ from \cref{th:analytical-lemma} as in the proof of \cref{th:wiener:stabilite:outline} to truncate the series and conclude for the convergence.
\end{proof}

We conclude this section with the proof of \cref{th:super-lemme}.
\begin{proof}[Proof of {\cref{th:super-lemme}}]
We proceed by induction on $p-s$. When $s=p-1$, the lemma \ref{th:wiener:normal-convergence} ensures a Gaussian limit for $F_n$ hence $Q$ must be linear in this case. Indeed $\left\lfloor\frac{p}{s+1}\right\rfloor=1$ which is consistent.
\medskip
As established above, for any sequence $(F_n)_{n\ge 1}$ in $\mathcal{W}_p$, one may write for any $r\ge 1$ that $F_n=\sum_{i=1}^{r} F_{n,r}+R_{n,r},$ where 
\begin{itemize}
\item \(F_{n,i}=\sum_{k=1}^{\lfloor p/q_i\rfloor}A_{k,n}^{(i)} H_k(X_{n,i})\),
\item $X_{n,i}\xrightarrow[n\to\infty]~\mathcal{N}(0,1)$ and $\|\Gamma\left[R_{n,i-1},X_{n,i}\right]\|_2\ge \rho_{q_i}(R_{n,i})-1/n$.
\end{itemize}
If one further assume that $\rho_s(F_n)\to 0$, for some $s\ge 1$, then one gets that $q_i\ge s+1$ in the above decomposition. Indeed, recall that $q_i\coloneq\min(k\ge 1|\rho_k(R_{n,i-1})=\rho_k(A_{0,n}^{(i-1)})\not\to 0)$ and that thanks to \cref{th:wiener:decomposition:remainder-gaussian} we get that $\rho_s(A_{k,n}^{(i)})\to 0$ for each $k\in\llbracket 0, \frac{p}{q_i}\rrbracket$ and each $i\ge 1$.  Besides, relying on the proof of our main Theorem, we may assert that for some suitable $r_n\to\infty$ we have $\rho_{s+1}\left(R_{n,r_n}\right)\to 0$ as well as (up to extracting subsequences)
\begin{itemize}
\item $F_{n,i}\to F_{\infty,i}\coloneq\sum_{k=1}^{\lfloor p/q_i\rfloor}A_{k,\infty}^{(i)}H_k(X_{\infty,i})$,
\item $R_{n,r_n}\to R_\infty$,
\item $F_n\to \sum_{i=1}^\infty F_{i,\infty}+R_{\infty}$ (the series is converging in $L^2$)
\end{itemize}
with all convergences being joint. By our induction hypothesis we may already infer that $R_\infty$ may be represented by a polynomial of degree less than $\lfloor p/(s+2)\rfloor$. Moreover, $X_{\infty,i}$ is an element of the first chaos and $H_k(X_{\infty,i})$ is then a degree $k$-polynomial. Besides, for each $i\ge 1$ and each $k\in\llbracket 1, p/q_i\rrbracket$, $A_{k,n}^{(i)}$ is asymptotically in $\mathcal{W}_{p-k q_i}$ with $\rho_s(A_{k,n}^{(i)})\to 0$. If $s\ge p-k q_i$ then we obtain that $A_{k,n}^{(i)}$ is asymptotically independent of a chaos of higher order  which entails that $A_{k,n}^{(i)}\to 0$. When $p-k q_i \ge s+1$, since $1\le p-k q_i-s <p-s$, we may use our induction hypothesis and $A_{k,\infty}^{(i)}$ may be represented by a polynomial of degree $\left\lfloor \frac{p-kq_i}{s+1}\right\rfloor$. Hence, $F_{i,\infty}$ is a polynomial of degree less than 
$$\max_{k\in\llbracket 1, \lfloor (p-s-1)/q_i\rfloor\rrbracket} \left\lfloor \frac{p-kq_i}{s+1}\right\rfloor +k\le \max_{k\in\llbracket 1, \lfloor (p-s-1)/q_i\rfloor\rrbracket} \left\lfloor \frac{p-k(s+1)}{s+1}\right\rfloor +k =\left\lfloor \frac{p}{s+1}\right\rfloor,$$
where the last inequality uses that $q_i\ge s+1$ and that $x\mapsto\lfloor x\rfloor$ is monotonic. The final argument uses the fact that the series $\sum_{i\ge 1} F_{\infty,i}$ is converging in $L^2$ and the fact $\mathcal{W}_{\le m}$ is a closed space for any $m\ge 0$.
\end{proof}

\section{From stability of Wiener chaoses to stability of arbitrary polynomial chaoses}

\begin{theorem}\label{th:stabilite-general-multidim}
  Let $(P_{i,n})_{i,n}$ be polynomials of degree at most $d \in \mathbb{N}$; let $\vec{X}$ be a vector of independent and identically random variables, that are centered, with unit variance, and admitting moments at every order; and let $\vec{F}_{n}$ be the infinite random vector given by
  \begin{equation*}
    F_{i,n} \coloneq P_{i,n}(\vec{X}), \qquad i,n \in \mathbb{N}.
\end{equation*}
Assume that $(\vec{F}_{n})$ converges in law.
Then, there exist polynomials $(Q_{i})_{i}$ of degree at most $d$ and an infinite random vector $\vec{Y}$ with independent entries such that $\Law{Y_{i}} \in \set*{ \Law{X_{1}}, \mathcal{N}(0,1)}$ and
\begin{equation*}
  \vec{F}_{n} \xrightarrow[n \to \infty]{\law} (Q_{i}(\vec{Y}))_{i}.
\end{equation*}
\end{theorem}

\subsection{Reduction to multilinear polynomials}\label{s:reduction-multilinear}
We build upon an invariance results from \cite{MOO}.
Let us introduce some notions taken from there.
In view of our assumptions, there exist polynomials $T_{0} \coloneq 1,T_{1} \coloneq x,T_{2} \coloneq x^{2}-1, \dots, T_{d}$, such that the random variables $T_{0}X_{1}, T_{1}X_{1}, \dots, T_{d}X_{1}$ are orthonormal.
Let us consider the \emph{orthonormal ensembles}
\begin{equation*}
  \mathscr{X}_{j} \coloneq \set*{ T_{0}X_{j}, T_{1}X_{j}, \dots, T_{d}X_{j}}.
\end{equation*}
Then the $F_{n,i}$'s are \emph{multilinear polynomials} over these ensembles, namely
\begin{equation*}
  F_{n,i} = \sum_{\alpha \in \{1, \dots, d\}^{\mathbb{N}}} a_{n,i,\alpha} \prod_{j \in \mathbb{N}} \mathscr{X}_{j,\alpha_{j}},
\end{equation*}
for suitably chosen real coefficients $(a_{n,i,\alpha})$.
Since all the $\mathscr{X}_{j}$'s actually contains the same number of variables, it is sufficient for our purpose to work with elementary multilinear polynomials of the form
\begin{equation}\label{eq:definition-multlinear}
  \sum_{J \subset \mathbb{N}} a_{J} Z_{J},
\end{equation}
where the $a_{J}$'s are real coefficients and $Z_{J} \coloneq \prod_{j \in J} Z_{j}$, with $Z_{j} \in \mathscr{X}_{j}$.

\subsection{Proof by induction}
\begin{definition}
  We say that a sequence of vectors $(\vec{F}_{n})$ is $d$-admissible provided that there exists $K \in \mathbb{N}$ such that for all $i$ and $n \in \mathbb{N}$, $F_{n,i}$ is a multilinear polynomial as in \cref{eq:definition-multlinear} with degree at most $K$, and
  \begin{enumerate}[(i)]
    \item either $\tau(F_{n,i}) \to 0$, as $n \to \infty$,
    \item or $\mathrm{deg}(F_{n,i}) \leq d$.
  \end{enumerate}
\end{definition}

\begin{definition}[Induction hypothesis]
  For all $d \in \mathbb{N}$, we write $\mathsf{H}(d)$ to indicate that the conclusion of \cref{th:stabilite-general-multidim} holds provided $(\vec{F}_{n})$ is a $d$-admissible sequence.
\end{definition}
In view of \cref{s:reduction-multilinear}, establishing $\mathsf{H}(d)$ for all $d \in \mathbb{N}$ proves \cref{th:stabilite-general-multidim}.

\subsection{Initialisation}
Take a $1$-admissible sequence $(\vec{F}_{n})$.
This means that for $n$ and $i \in \mathbb{N}$ either $F_{n,i}$ is linear, either $\tau(F_{n,i}) \to 0$ as $n \to \infty$.
\subsubsection{Decomposition of the linear terms}
Let us write $I$ for the set of indices $i$ such that for all $n \in \mathbb{N}$ $F_{n,i}$ is linear.
Since we can discard the constants that do not play a role in the argument, we can write for all $i \in I$,
\begin{equation*}
  F_{n,i} = \sum_{j \in \mathbb{N}} a_{n,i,j} Z_{j},
\end{equation*}
for some real numbers $a_{n,i,j}$.
Choose a permutation $\sigma_{n,i}$ that order the $(a_{n,i,j})_{j}$ in module, namely we require that
\begin{equation*}
  |a_{n,i,\sigma_{n,i}(1)}| \geq |a_{n,i,\sigma_{n,i}(2)}| \geq ...
\end{equation*}
Without loss of generality, we assume that $\sigma_{n,1} = \mathrm{id}$.
Up to extracting non-relabelled subsequences, we have that
\begin{equation*}
  a_{n,i,\sigma_{n,i}(j)} \xrightarrow[n \to \infty]{} a_{\infty,i,j},
\end{equation*}
for some real numbers $a_{\infty,i,j}$.
Since the elements of the sequences are ordered, we find that the convergence holds in $\ell^{\infty}(\mathbb{N})$. Indeed, we first notice that we have
\begin{itemize}
\item $\forall j\ge 1$, $a_{n,i,\sigma_{n,i}(j)}^2\le \frac{a_{n,i,\sigma_{n,i}(1)}^2+\cdots+a_{n,i,\sigma_{n,i}(j)}^2}{j}\le \frac{1}{j}$,
\item $\forall j\ge 1$, $\left|a_{n,i,\sigma_{n,i}(j)}-a_{\infty,i,j}\right|\to 0$.
\end{itemize}
Gathering these two facts entails the $\ell^{\infty}$-convergence.

\subsubsection{Terms with small influence}
In particular, letting
\begin{equation*}
  R_{n,i} \coloneq
  \begin{dcases}
    \sum_{j \in \mathbb{N}} (a_{n,i,\sigma_{n,i}(j)} - a_{\infty,i,j}) Z_{\sigma_{n,i}(j)}, & i \in I;
    \\ F_{n,i}, & i \not \in I.
  \end{dcases}
\end{equation*}
we find that $\tau(R_{n,i}) \to 0$ as $n \to \infty$.
Take $\vec{G} = (G_{k})$ a standard Gaussian vector.
Our first result characterizes the limit in law for elements of low influence following a combination of an invariance principle \cite{RotarPolylinear,MOO} together with our stability result \cref{th:stabilite-wiener-multidim}.
\begin{lemma}
  There exist polynomials $(T_{\infty,i})$ with $\mathrm{deg}(T_{\infty,i}) \leq K$ for all $i \in \mathbb{N}$ and $\mathrm{deg}(T_{\infty,i}) \leq 1$ for $i \in I$ such that
\begin{equation}\label{eq:initialisation:low-influence:convergence}
  \paren*{ R_{n,i} : i \in \mathbb{N} } \xrightarrow[n \to \infty]{\law} \paren*{T_{\infty,i}(\vec{G}) : i \in \mathbb{N} }.
\end{equation}
\end{lemma}
\begin{proof}
  Take polynomials $(T_{n,i})$ such that $R_{n,i} = T_{n,i}(\vec{Z})$.
  By \cref{th:stabilite-wiener-multidim}, there exist polynomials $T_{\infty,i}$ as in the claim such that
  \begin{equation}\label{eq:low-influence:gaussian}
    \paren*{ T_{n,i}(\vec{G}) : i \in I } \xrightarrow[n \to \infty]{\law} \paren*{ T_{\infty,i}(\vec{G}) : i \in I  }
  \end{equation}
  Take $M \in \mathbb{N}^{*}$.
  Consider the Wasserstein distance $\mathbf{W}_{2}$ defined in \cref{eq:def-wasserstein}, by the triangle inequality
  \begin{equation}\label{eq:low-influence:wasserstein}
    \mathbf{W}_{2}\paren*{ \begin{pmatrix}R_{n,1}\\\dots\\R_{n,M}\end{pmatrix}, \begin{pmatrix}T_{\infty,1}(\vec{G})\\\dots\\T_{\infty,M}(\vec{G})\end{pmatrix} } \leq
    \mathbf{W}_{2}\paren*{ \begin{pmatrix}R_{n,1}\\\dots\\R_{n,M}\end{pmatrix}, \begin{pmatrix}T_{n,1}(\vec{G})\\\dots\\T_{n,M}(\vec{G})\end{pmatrix} } +
    \mathbf{W}_{2}\paren*{ \begin{pmatrix}T_{n,1}(\vec{G})\\\dots\\T_{n,M}(\vec{G})\end{pmatrix}, \begin{pmatrix}T_{\infty,1}(\vec{G})\\\dots\\T_{\infty,M}(\vec{G})\end{pmatrix} }.
  \end{equation}
  Since the Wasserstein distance $\mathbf{W}_{2}$ metrizes the convergence in law for random polynomials vectors in $\mathbb{R}^{M}$ (see \cref{th:wasserstein-polynomial}), the second term in the right hand side vanishes as $n \to \infty$.
  On the other hand by the multivariate version of the invariance principle from \cite{MOO}, see \cite[Thm.~7.1]{NPR}, the first term is controlled, up to a constant, by $\max_{i=1,\dots,M} \tau(T_{n,i})$.
  Since by construction this influence goes to $0$, we conclude that the left-hand side in \cref{eq:low-influence:wasserstein} goes to $0$.
  Since this holds for an arbitrary $M \in \mathbb{N}^{*}$, the proof is complete.
\end{proof}

\subsubsection{The \texorpdfstring{$\ell^{2}$}{l2}-term}
Here, our goal is to understand the asymptotic as $n \to \infty$ of the linear terms
\begin{equation*}
  L_{n,i} \coloneq \sum_{j \in \mathbb{N}} a_{\infty,i,j} Z_{\sigma_{n,i}(j)}.
\end{equation*}
Since the coefficients $a_{\infty,i,j}$'s are already independent of $n$, we study the limit in law of $(Z_{\sigma_{n,i}(j)})$.
Define
\begin{equation*}
  C_{n}((i,j), (\tilde\imath, \tilde\jmath)) \coloneq \Esp*{ Z_{\sigma_{n,i}(j)} Z_{\sigma_{n,\tilde\imath}(\tilde\jmath)} } \in \{0,1\}.
\end{equation*}
Up to non-relabelled extractions, we have that $C_{n}$ converges pointwise to some $C_{\infty}$.
Define the relation $(i,j) \sim (\tilde\imath,\tilde\jmath)$ provided $C((i,j),(\tilde\imath,\tilde\jmath)) = 1$, which is equivalent to: for $n$ large enough $Z_{\sigma_{n,i}(j)} = Z_{\sigma_{n,\tilde\imath}(\tilde\jmath)}$.
From this representation, we see that $\sim$ is an equivalence relation.
Up to non-relabelled extraction, we can consider $\mu_{\infty,i,j}$ the limiting distribution of $Z_{\sigma_{n,i}(j)}$, which, by definition of $\sim$, only depends on the equivalence class of $(i,j)$.
Since the $Z_{i}$'s are independent and $\mathscr{Z} \coloneq \set*{ \Law*{Z_{i}} : i \in \mathbb{N} }$ is a finite set, we actually have that $\mu_{\infty,i,j} \in \mathscr{Z}$
Let us consider a sequence of independent random variables
\begin{equation*}
  \paren*{ Z_{\infty,C} : C \in (\mathbb{N}^{2} \setminus \sim) },
\end{equation*}
such that $\Law*{Z_{\infty,C}} = \mu_{\infty,i,j}$ for any $(i,j) \in C$.
We build a $\mathbb{N}^{2}$-indexed sequence from there by setting
\begin{equation*}
  Z_{\infty,i,j} \coloneq Z_{\infty,C}, \qquad (i,j) \in C,\, C \in (\mathbb{N}^{2} \setminus \sim).
\end{equation*}
Our construction guarantees that
\begin{equation}\label{eq:l2:convergence-Z}
  (Z_{\sigma_{n,i}(j)} : (i,j) \in \mathbb{N}^{2}) \xrightarrow[n \to \infty]{\law} (Z_{\infty,i,j} : (i,j) \in \mathbb{N}^{2}).
\end{equation}
Define
\begin{equation*}
  L_{\infty,i} \coloneq \sum_{j \in \mathbb{N}} a_{\infty,i,j} Z_{\infty,i,j},
\end{equation*}
which exists as a series converging in $L^{2}(\prob)$.
Indeed, since $\sigma_{n,i}$ is a bijection, it is impossible to have $(i,j) \sim (i, \tilde\jmath)$ for $j \ne \tilde\jmath$.
Thus all the terms in the series are actually orthogonal and the series is convergent since $\sum_{j \in \mathbb{N}} a_{\infty,i,j}^{2} \leq 1$.
We can now characterize the limit in law of the $L_{n,i}$'s.
\begin{lemma}
  With the above notations
  \begin{equation*}
    (L_{n,i} : i \in \mathbb{N}) \xrightarrow[n \to \infty]{\law} (L_{\infty,i} : i \in \mathbb{N} ).
  \end{equation*}
\end{lemma}
\begin{proof}
  For an integer $Q$ and $n \in \mathbb{N}$, write
\begin{equation*}
  L_{n,i}^{\leq Q} \coloneq \sum_{j=0}^{Q} a_{\infty,i,j} Z_{\sigma_{n,i}(j)}, \qquad \text{and} \qquad
L_{n,i}^{> Q} \coloneq \sum_{j=Q+1}^{\infty} a_{\infty,i,j} Z_{\sigma_{n,i}(j)}.
\end{equation*}
We use a similar notation $L_{\infty,i}^{\leq Q}$ and $L_{\infty,i}^{>Q}$.
Let $M \in \mathbb{N}^{*}$ and $\varepsilon > 0$.
Take $Q$ such that
\begin{equation*}
  \max_{i =1,\dots,M} \sum_{j=Q+1}^{\infty} a_{\infty,i,j}^{2} \leq \varepsilon.
\end{equation*}
This ensures that
\begin{equation}\label{eq:l2:small-variance}
  \sup_{n \in \mathbb{N}} \max_{i=1,\dots,M} \Var*{ L_{n,i}^{>Q} } \leq \varepsilon, \qquad \text{and} \qquad \max_{i=1,\dots,M} \Var*{L_{\infty,i}^{>Q}} \leq \varepsilon.
\end{equation}
Consider the Wasserstein distance $\mathbf{W}_{2}$ defined in \cref{eq:def-wasserstein}.
By \cref{eq:l2:small-variance}, we find that, uniformly in $n$,
\begin{equation*}
  \mathbf{W}_{2}\paren*{ \begin{pmatrix}L_{n,1}\\\dots\\L_{n,M}\end{pmatrix}, \begin{pmatrix}L_{n,1}^{\leq Q}\\\dots\\L_{n,M}^{\leq Q}\end{pmatrix}} + \mathbf{W}_{2}\paren*{ \begin{pmatrix}L^{\leq Q}_{\infty,1}\\\dots\\L_{\infty,M}^{\leq Q}\end{pmatrix}, \begin{pmatrix}L_{\infty,1}\\\dots\\L_{\infty,M}\end{pmatrix} } \leq 2 M \varepsilon.
\end{equation*}
On the other hand, \cref{eq:l2:convergence-Z} implies that
\begin{equation*}
  \mathbf{W}_{2}\paren*{\begin{pmatrix}L_{n,1}^{\leq Q}\\\dots\\L_{n,M}^{\leq Q}\end{pmatrix}, \begin{pmatrix}L_{\infty,1}^{\leq Q}\\\dots\\L_{\infty,M}^{\leq Q}\end{pmatrix} } \xrightarrow[n \to \infty]{} 0.
\end{equation*}
Since $\varepsilon$ is arbitrary, we actually find that $(L_{n,i} : i=1,\dots,M)$ converges in Wasserstein distance to $(L_{\infty,i} : i=1,\dots,M)$.
Since the Wasserstein distance metrizes the convergence in law \cref{th:wasserstein-polynomial}, and that the convergence holds for all $M \in \mathbb{N}^{*}$, this completes the proof.
\end{proof}

\subsubsection{Asymptotic independence and joint convergence}
So far we have established the convergence of the low influence $R_{n,i}$'s and the linear $L_{n,i}$'s separately.
Let us show that the two sequences are asymptotically independent.
\begin{lemma}\label{th:initialisation:asymptotic-independence}
We further assume that $\vec{G}$ and $(Z_{\infty,i,j})$ are independent.
Then,
\begin{equation*}
  \paren*{ L_{n,i}, R_{n,i} : i \in \mathbb{N} } \xrightarrow[n \to \infty]{\law} \paren*{ L_{\infty,i}, R_{\infty,i} : i \in \mathbb{N} }.
\end{equation*}
\end{lemma}

\begin{proof}
Fix $M \in \mathbb{N}^{*}$ take $\Phi, \Psi \in \mathscr{C}^{\infty}_{b}(\mathbb{R}^{M})$.
As above take $\varepsilon > 0$ and $Q \in \mathbb{N}^{*}$ such that \cref{eq:l2:small-variance} holds.
Thus by a Taylor expansion, we find that
\begin{equation}\label{eq:joint-convergence:small-variance}
  \Esp*{ \Phi((L_{n,i}))_{i \leq M} \Psi((R_{n,i})_{i \leq M}) } = \Esp*{ \Phi((L_{n,i}^{\leq Q})_{i\leq M} ) \Psi((R_{n,i})_{i\leq M}) } + O(M \varepsilon).
\end{equation}
The random variables $(L_{n,i}^{\leq Q})_{i \leq M}$ only depend on the $l \coloneq (Q+1)(M+1)$ random variables $(Z_{\sigma_{n,i}(j)} : i \leq M, j \leq Q)$.
Let us write $(\hat{R}_{n,i})$ for the corresponding random variables where these $l$ random variables are replaced by independent copies.
\begin{equation}\label{eq:joint-convergence:small-influence}
  \Esp*{ \Phi((L_{n,i}^{\leq Q}))_{i \leq M} \Psi((R_{n,i})_{i \leq M}) } = \Esp*{ \Phi((L_{n,i}^{\leq Q})_{i\leq M} ) \Psi((\hat{R}_{n,i})_{i\leq M}) } + O(M\tau_{n}),
\end{equation}
where $\tau_{n} \coloneq \max_{i=1,\dots,M} \tau(R_{n,i}) \to 0$ as $n \to \infty$.
Since the two random vectors on the right-hand side are now independent, combining \cref{eq:joint-convergence:small-influence,eq:joint-convergence:small-influence} we find that
\begin{equation*}
  \begin{split}
    \Esp*{ \Phi((L_{n,i}))_{i \leq M} \Psi((R_{n,i})_{i \leq M}) } &= \Esp*{ \Phi((L_{n,i}^{\leq Q})_{i\leq M} ) } \Esp*{ \Psi((\hat{R}_{n,i})_{i\leq M}) } + O(M (\varepsilon + \tau_{n})) 
                                                                 \\&= \Esp*{ \Phi((L_{n,i})_{i \leq M}) } \Esp*{ \Psi((R_{n,i})_{i \leq M}) } + O(M(\varepsilon + \tau_{n})).
  \end{split}
\end{equation*}
  Since $\varepsilon$ was arbitrary, we obtain, by letting $n \to \infty$, the asymptotic independence for all vector of finite length $M$.
  Since $M$ is also arbitrary, we conclude.
\end{proof}

\subsubsection{Conclusion}
By the continuous mapping theorem, \cref{th:initialisation:asymptotic-independence} ensures that
\begin{equation*}
  (F_{n,i} : i \in \mathbb{N}) \xrightarrow[n \to \infty]{\law} (L_{\infty,i} + R_{\infty,i} : i \in \mathbb{N}),
\end{equation*}
where we recall that the Gaussian vector $\vec{G}$ from which the $R_{\infty,i}$'s are constructed, and the array $(Z_{\infty,i,j})_{i,j}$ from which the $L_{\infty,i}$ are constructed are independent. 
To complete the proof, we now construct explicitly $\vec{Y}$ and the $Q_{i}$'s such that the conclusion of \cref{th:stabilite-general-multidim} holds. 
Take any bijection $\varphi \colon \mathbb{N}^{2} \to \mathbb{N}$.
Define $\vec{Y} = (Y_{k})$ with
\begin{equation*}
  Y_{2k} \coloneq G_{k}, \qquad \text{and} \qquad
Y_{2k+1} \coloneq Z_{\infty,\varphi^{-1}(k)}.
\end{equation*}
Recall that $R_{\infty,i} = T_{\infty,i}(\vec{G})$ for some polynomials $T_{\infty,i}$ which we can immediately write as $Q_{i,1}(\vec{Y})$.
Defining $\mathbb{N}_{i} \coloneq \varphi(\{i\} \times \mathbb{N})$, we have
\begin{equation*}
  L_{\infty,i} = \sum_{k \in \mathbb{N}_{i}} a_{\infty,\varphi^{-1}(k)} Y_{2k+1} \coloneq Q_{i,2}(\vec{Y}).
\end{equation*}
Setting $Q_{i} \coloneq Q_{i,1} + Q_{i,2}$ we conclude.

\subsection{Induction step}
Let us assume $\mathsf{H}(d)$ for some $d \in \mathbb{N}$, and let us show $\mathsf{H}(d+1)$.
Take $(\vec{F}_{n})$ a $(d+1)$-admissible sequence.
The idea is to build a $d$-admissible vector out of $\vec{F}_{n}$ and use our induction hypothesis.
Let $I_{0} \coloneq \set*{ i \in \mathbb{N} : \tau(F_{n,i}) \to 0 }$.
Since elements with vanishing influence are for free in the definition of admissibility we keep $F_{n,i}$ as is for $i \in I_{0}$.
\subsubsection{Decomposing the terms of large influence}
For $i \not\in I_{0}$, we consider permutations $\sigma_{n,i} \colon \mathbb{N} \to \mathbb{N}$ that orders the influences, namely
\begin{equation*}
  \tau_{\sigma_{n,i}(1)}(F_{n,i}) \geq \tau_{\sigma_{n,i}(2)}(F_{n,i}) \geq \dots
\end{equation*}
Our strategy consists in successively removing variables with large influence.
To that extent, define
\begin{equation*}
  \mathscr{J}_{n,i,l} \coloneq \set*{ J \subset \mathbb{N} : \forall k \in \{1, \dots,l-1\},\sigma_{n,i}(k) \not\in J, \ \text{and}\ \sigma_{n,i}(l) \in J };
\end{equation*}
In other words, $\mathscr{J}_{n,i,l}$ is the set of all subsets where the $l$-th largest indices, in terms of influence, appears but not the indices with larger influence.
It is thus natural to define the polynomial with the $l$-th influence removed
\begin{equation*}
  R_{n,i,l} \coloneq \sum_{J \in \mathscr{J}_{n,i,l}} a_{J} Z_{J \setminus \sigma_{n,i}(l)};
\end{equation*}
  as well as the reminder
\begin{align*}
S_{n,i,p} \coloneq \sum_{J \in \cap_{l \leq p} \mathscr{J}_{n,i,l}^{C}} a_{J} Z_{J}.
\end{align*}
By this construction, we find that for all $p \in \mathbb{N}^{*}$
\begin{equation}\label{eq:induction:decomposition-influence}
  F_{n,i} = Z_{\sigma_{n,i}(1)} R_{n,i,1} + \dots + Z_{\sigma_{n,i}(p)} R_{n,i,p} + S_{n,i,p}.
\end{equation}
Since we remove the $p$-th first largest influence and that $F_{n,i}$ is normalized, we also find that
\begin{equation}\label{eq:induction:reminder-small-influence}
  \tau(S_{n,i,p}) \leq O(1/p), \qquad \text{uniformly in $n$ and $i$}.
\end{equation}

\subsubsection{The admissible vector and its convergence}
Take $(p_{n})$ be an increasing sequence of integers converging to $\infty$, to be specified later.
By \cref{eq:induction:reminder-small-influence}, we find that $\tau(S_{n,i,p_{n}}) \to 0$.
On the other hand, the $Z_{\sigma_{n,i}(j)}$'s are polynomials of degree $\leq d$ --- actually, they are of degree exactly $1$.
Similarly, all the $R_{n,i,j}$'s are, by construction, of degree at most $d$.
Thus, any vector containing them is a $d$-admissible vector.
To be concrete, choose a bijection $\varphi \colon \mathbb{N}^{*} \times \mathbb{Z} \to \mathbb{N}$.
Consider the array $A_{n} = (A_{n}(i,j))_{i \in \mathbb{N}^{*}, j \in \mathbb{Z}}$ such that
\begin{align*}
  & A_{n}(i, 0) \coloneq S_{n,i,p_{n}};
\\& A_{n}(i, j) \coloneq R_{n,i,j}, \qquad j \in \mathbb{N}^{*};
\\& A_{n}(i, j) \coloneq Z_{\sigma_{n,i}(j)}, \qquad j \in \mathbb{Z}_{-}^{*};
\end{align*}
and define the vector
\begin{equation*}
  \vec{V}_{n} \coloneq (A_{n}(\varphi^{-1}(k)) : k \in \mathbb{N}).
\end{equation*}
Thus, $(\vec{V}_{n})$ is $d$-admissible and by the induction hypothesis, there exist $\tilde{Q}_{i}$ of degree at most $d$ and a vector $\tilde{Y}$ such that
\begin{equation*}
  \vec{V}_{n} \xrightarrow[n \to \infty]{\law} \vec{V}_{\infty} \coloneq (\tilde{Q}_{i}(\vec{Y})).
\end{equation*}

\subsubsection{Identification of the limit of \texorpdfstring{$\vec{F}_{n}$}{Fn}}
We can subsequently define
\begin{align*}
  & Z_{\infty,i,j} \coloneq V_{\infty,\varphi(i,-j)};
\\& R_{\infty,i,j} \coloneq V_{\infty,\varphi(i,j)};
\\& S_{\infty,i} \coloneq V_{\infty,\varphi(i,0)};
\\& F_{\infty,i,j} \coloneq Z_{\infty,i,j} R_{\infty,i,j}.
\end{align*}
Since we are working with polynomials, by \cref{th:polynomials:convergence-moment}, we find that for all $i \in \mathbb{N}$, the terms of $(F_{\infty,i,j})_{j}$ are orthogonal and satisfy
\begin{equation*}
  \sum_{j=1}^{\infty} \Esp*{F_{\infty,i,j}^{2}} \leq 1.
\end{equation*}
In particular, we find that the following random variables is well-defined
\begin{equation*}
  F_{\infty,i} \coloneq \sum_{j=1}^{\infty} F_{\infty,i,j} + S_{\infty,i},
\end{equation*}
since the series converges in $L^{2}(\prob)$.
Since the space of polynomials is also closed in $L^{2}(\prob)$ \cref{th:polynomials:closed-L2}, we find that $F_{\infty,i} = Q_{i}(\vec{Y})$ for some $Q_{i}$ of degree at most $d$.

\subsubsection{Conclusion}
To conclude, let us show the following.
\begin{equation*}
  \vec{F}_{n} \xrightarrow[n \to \infty]{\law} \vec{F}_{\infty}.
\end{equation*}

Again we need to show that despite the infinite series, the infinite series is not a problem.
For this we use a truncation argument already presented in the initialisation case that we do not repeat.

\printbibliography%
\end{document}